\documentclass[11pt, oneside]{article}   	
\usepackage{geometry}                		
\usepackage[dvipdfmx]{graphicx}
\usepackage[usenames]{xcolor}		
\usepackage{amsmath,amsthm,amssymb}
\usepackage{amscd}

\theoremstyle{definition}
\newtheorem{dfn}{Definition}[section]
\newtheorem{rem}[dfn]{Remark}
\newtheorem{exple}[dfn]{Example}

\theoremstyle{plain}
\newtheorem{thm}[dfn]{Theorem}
\newtheorem{prop}[dfn]{Proposition}
\newtheorem{lem}[dfn]{Lemma}
\newtheorem{cor}[dfn]{Corollary}

\newcommand{\K}{\mathbb{K}}
\newcommand{\Z}{\mathbb{Z}}

\newcommand{\der}{\mathrm{Der}}

\newcommand{\tder}{\mathrm{tDer}}
\newcommand{\sder}{\mathrm{sDer}}
\newcommand{\krv}{\mathrm{krv}}
\newcommand{\D}{\mathrm{D}}
\newcommand{\tD}{\mathrm{tD}}

\newcommand{\taut}{\mathrm{TAut}}
\newcommand{\saut}{\mathrm{SAut}}
\newcommand{\KRV}{\mathrm{KRV}}

\newcommand{\mult}{\mathrm{mult}}
\newcommand{\add}{\mathrm{add}}
\newcommand{\KKS}{\mathrm{KKS}}

\renewcommand{\div}{\mathrm{div}}
\newcommand{\ddiv}{\mathrm{Div}}
\newcommand{\tdiv}{\mathrm{tDiv}}
\newcommand{\tJ}{\mathrm{tJ}}

\newcommand{\pa}{\partial}
\newcommand{\tde}{\widetilde{\Delta}}
\newcommand{\bch}{\mathrm{bch}}
\newcommand{\alg}{\mathrm{alg}}
\newcommand{\op}{\mathrm{op}}
\newcommand{\ad}{\mathrm{ad}}

\newcommand{\alt}{\mathrm{Alt}}
\newcommand{\rot}{\mathrm{rot}}
\newcommand{\even}{\mathrm{even}}
\newcommand{\odd}{\mathrm{odd}}

\title{The Goldman-Turaev Lie bialgebra in genus zero and the Kashiwara-Vergne problem}
\author{Anton Alekseev\thanks{Department of Mathematics, University of Geneva, 2-4 rue du Lievre, 1211 Geneva, Switzerland \texttt{e-mail:anton.alekseev@unige.ch}} \and Nariya Kawazumi\thanks{Department of Mathematical Sciences, University of Tokyo, 3-8-1 Komaba, Meguro-ku, Tokyo 153-8914, Japan \texttt{e-mail:kawazumi@ms.u-tokyo.ac.jp}}
\and Yusuke Kuno\thanks{Department of Mathematics, Tsuda College, 2-1-1 Tsuda-machi, Kodaira-shi, Tokyo 187-8577, Japan \texttt{e-mail:kunotti@tsuda.ac.jp}} \and Florian Naef\thanks{Department of Mathematics, University of Geneva, 2-4 rue du Lievre, 1211 Geneva, Switzerland \texttt{e-mail:florian.naef@unige.ch}}}
\date{}							

\begin{document}

\maketitle

\begin{abstract}
In this paper, we describe a surprising  link between the theory of the Goldman-Turaev Lie bialgebra on surfaces of genus zero and the Kashiwara-Vergne (KV) problem in Lie theory. Let $\Sigma$ be an oriented 2-dimensional manifold with non-empty  boundary and $\mathbb{K}$ a field of characteristic zero. The Goldman-Turaev Lie bialgebra is defined by the Goldman bracket $\{ -,- \}$ and Turaev cobracket $\delta$  on the $\mathbb{K}$-span of homotopy classes of free loops on $\Sigma$. 

Applying an expansion $\theta: \mathbb{K}\pi \to \mathbb{K}\langle x_1, \dots, x_n \rangle$ yields an algebraic description of the operations $\{ -,- \}$ and $\delta$ in terms of non-commutative variables $x_1, \dots, x_n$. If $\Sigma$ is a surface of genus $g=0$ the lowest degree parts $\{ -,- \}_{-1}$ and $\delta_{-1}$ are canonically defined (and independent of $\theta$). They define a Lie bialgebra structure on the space of cyclic words which was introduced and studied by Schedler \cite{Schedler}. It was conjectured by the second and the third authors that one can define an expansion $\theta$ such that $\{ -,- \}=\{ -,- \}_{-1}$ and $\delta=\delta_{-1}$. The main result of this paper states that for surfaces of genus zero constructing such an expansion is essentially equivalent to the KV problem. In \cite{Mas15}, Massuyeau constructed such expansions using the Kontsevich integral.

In order to prove this result, we show that the Turaev cobracket $\delta$ can be constructed in terms of the double bracket (upgrading the Goldman bracket) and the non-commutative divergence cocycle which plays the central role in the KV theory. Among other things, this observation gives a new topological interpretation of the KV problem and allows to extend it to surfaces with arbitrary number of boundary components (and of arbitrary genus, see \cite{short}).
\end{abstract}
{\bf keywords:} Goldman bracket, Turaev cobracket, Kashiwara-Vergne conjecture

\section{Introduction}

Let $\Sigma$ be a 2-dimensional oriented manifold of genus $g$ with non empty boundary and let $\mathbb{K}$ be a field of characteristic zero. In \cite{Go86}, Goldman introduced a Lie bracket on the $\mathbb{K}$-span of homotopy classes of free loops on $\Sigma$. One can also think of this space as a space spanned (over $\mathbb{K}$) by conjugacy classes in the fundamental group $\pi:=\pi_1(\Sigma)$. The notation that we use for this space is
$$
|\mathbb{K}\pi| = \mathbb{K}\pi/[\mathbb{K}\pi, \mathbb{K}\pi].
$$
One of the {\em raisons d'\^etre} of the Goldman bracket is the finite dimensional description of the Atiyah-Bott symplectic structure on moduli spaces of flat connections on $\Sigma$ \cite{Atiyah-Bott}.
In \cite{Tu91}, Turaev constructed a Lie cobracket on the quotient space $|\mathbb{K}\pi|/\mathbb{K}\, {\bf 1}$, where ${\bf 1}$ stands for the homotopy class of a trivial (contractible) loop. Given a framing (that is, a trivialization of the tangent bundle) of $\Sigma$, the cobracket $\delta$ can be lifted to a cobracket $\delta^+$ on $|\mathbb{K}\pi|$
which together with the Goldman bracket defines a Lie bialgebra structure. In more detail, this means that $\delta^+$ verifies the co-Jacobi identity, and that it is a 1-cocycle with respect to the Goldman bracket. This structure was studied in detail in the works of Chas and Sullivan, and it was one of the motivations for introducing the string topology program \cite{CS99}.

The Goldman bracket admits an upgrade to the double bracket in the sense of van den Bergh, $\kappa: \mathbb{K}\pi \otimes \mathbb{K}\pi \to \mathbb{K}\pi \otimes \mathbb{K}\pi$, and 
the  cobracket $\delta^+$ (for a given framing) can be upgraded to a map $\mu: \mathbb{K}\pi \to |\mathbb{K}\pi| \otimes  \mathbb{K}\pi$. All the operations described above: the Goldman bracket, the Turaev cobracket, the cobracket $\delta^+$, $\kappa$ and $\mu$ are defined in terms of intersections and self-intersections of curves on $\Sigma$.

Since $\partial \Sigma \neq \emptyset$, the fundamental group $\pi$ is a free group and one can consider expansions (that is, Hopf algebra homomorphisms) $\theta: \mathbb{K}\pi \to A$, where $A$ is a degree completed free Hopf algebra with generators $x_1, \dots, x_n$. An example of such an expansion is the exponential expansion $\theta^{\rm exp}(\gamma_i)=e^{x_i}$ which depends on the choice of the basis $\gamma_1, \dots, \gamma_n$ in $\pi$. An expansion allows to transfer the topologically defined operations $\kappa$ and $\mu$ (as well as the Goldman bracket and Turaev cobracket) to the free Hopf algebra $A$ and to the space of cyclic words in letters $x_1, \dots x_n$
$$
|A| = A/[A, A].
$$
We use the notation $a \mapsto |a|$ for the natural projection $A \to |A|$.
We will denote the operations transferred to $A$ and $|A|$ by $\kappa_\theta, \mu_\theta, \{ -, - \}_\theta, \delta^+_\theta$ and $\delta_\theta$, respectively. 

Assume that $\Sigma$ is a surface of genus $g=0$.
Consider the grading on $A$ defined by assignment ${\rm deg}(x_i)=1$ and the induced grading on $|A|$. The Lie bracket $\{ -, -\}_\theta: |A| \otimes  |A| \to |A|$ contains contributions of different degrees starting with degree $(-1)$. This lowest degree part is canonically defined and does not depend on the choice of $\theta$. The Goldman bracket being related to the Atiyah-Bott symplectic structures on moduli of flat connections, its lowest degree part comes from the Kirillov-Kostant-Souriau (KKS) Poisson structure on coadjoint orbits. Hence, we use the notation $\{ -, -\}_{\rm KKS}$ for this bracket.  For all expansions $\theta$ we have 
$$
\{ -, -\}_\theta= \{-, -\}_{\rm KKS} + \, \text{higher order terms}.
$$
It turns out that there is a class of expansions, called {\em special expansions}, for which all higher order terms vanish and the equality 
$$
\{ -, -\}_\theta= \{-, - \}_{\rm KKS}
$$
holds without extra corrections. Originally, the problem of finding such expansions was solved in the case of $\Sigma$ of genus $g$ with one boundary component in \cite{KK14}. This result was extended to the case of surfaces with arbitrary number of boundary components in \cite{KK16}, \cite{MTpre}.

In the case of $\Sigma$ of genus $g=0$ with 3 boundary components, the fundamental group $\pi$ is a free group with 2 generators and $A$ is the degree completion of the free Hopf algebra with two generators $x_1$ and $x_2$. In this case, special expansions are of the form $\theta_F = F^{-1} \circ \theta^{\exp}$, where $F$ is an automorphism of a free Lie algebra with two generators (naturally extended to $A$) with the property
\begin{equation} \label{eq:intro_KV0}
F(x_i)=e^{-a_i}x_ie^{a_i}
\end{equation}
for $i=1,2$ and
\begin{equation}  \label{eq:intro_KV1}
F(x_1 + x_2) =\log(e^{x_1}e^{x_2}).
\end{equation}

Similarly to the Goldman bracket, the transferred cobracket $\delta^+_\theta$ admits a decomposition
$$
\delta^+_\theta= \delta^{\rm alg} + \, \text{higher order terms}.
$$
For a surface of genus zero, $\delta^{\rm alg}$ is again an operator of degree $(-1)$. On the space of cyclic words, there is a Lie bilagebra structure defined by  the KKS bracket and the cobracket $\delta^{\rm alg}$. This Lie bialgebra structure was introduced and studied by Schedler \cite{Schedler}.

It is natural to ask a question of whether there exists a special  expansion $\theta$ such that 
$$
\delta^+_\theta= \delta^{\rm alg}.
$$
In this paper, we answer this question for surfaces of genus zero. The key observation is that equation \eqref{eq:intro_KV1} plays an important role in the Kashiwara-Vergne (KV) problem in Lie theory \cite{kashiwara-vergne}. In more detail, the reformulation of the KV problem in \cite{AT12} requires to find an automorphism $F$ of a free Lie algebra in two generators which satisfies \eqref{eq:intro_KV0}, \eqref{eq:intro_KV1} and the mysterious equation
\begin{equation} \label{eq:intro_KV2}
j(F^{-1})=|h(x_1)+h(x_2)-h(x_1+x_2)|.
\end{equation}
Here $h(z) \in \mathbb{K}[[z]]$ is a formal power series in one variable (also called {\em the Duflo function}). The map  $j$ is the group 1-cocycle integrating the non-commutative divergence 1-cocycle ${\rm div}: \tder(A) \to |A|$, and $\tder(A) \subset \der(A)$ is the Lie algebra of tangential derivations on $A$ which plays an important role in the Kashiwara-Vergne theory \cite{AT12}. The divergence cocycle admits an upgrade $\tdiv: \tder(A) \to |A| \otimes |A|$.

The following theorem establishes a surprising link between the Kashiwara-Vergne theory and the properties of the Goldman-Turaev Lie bialgebra:

\begin{thm}   \label{thm:intro_KV}
Let $F$ be a solution of equations \eqref{eq:intro_KV0} and \eqref{eq:intro_KV1}. Then, $\delta^+_\theta= \delta^{\rm alg}$ for the expansion $\theta = F^{-1} \circ \theta^{\rm exp}$ 
if and only if $F$ verifies equation  \eqref{eq:intro_KV2} up to linear terms.
\end{thm}
We also define a version of the KV problem for surfaces of genus zero and an arbitrary number of boundary components $n \geq 3$. It turns out that all these problems can be easily solved using the $n=3$ case. Moreover, a result analogous to Theorem \ref{thm:intro_KV} holds true for any $n$.

The operations $\kappa_\theta$ and $\mu_\theta$ also acquire a nice form for expansions $\theta$ defined by solutions of the Kashiwara-Vergne problem. In more detail, we have
$$
\kappa_\theta = \kappa_{\rm KKS} + \kappa_s,
$$
where $\kappa_{\rm KKS}$ is again the degree $(-1)$ part of $\kappa_\theta$. The double bracket $\kappa_s$ contains certain higher order terms encoded in the power series $s(z) \in \mathbb{K}[[z]]$ related to the even part $h_{\rm even}$ of the Duflo function by the following formula
$$
\frac{dh_{\rm even}}{dz} = \frac{1}{2}\left( s(z) + \frac{1}{2}\right) = -\frac{1}{2} \sum_{k\geq 1}
\frac{B_{2k}}{ (2k)!} z^{2k-1},
$$
where $B_{2k}$ are Bernoulli numbers.

There is a similar description for $\mu$:
$$
\mu_\theta=\mu^{\rm alg}+\mu_s + \tilde{\mu}_g,
$$
where $\mu_s$ is defined by the formal power series $s(z)$. The new contribution $\tilde{\mu}_g$ is defined by the formal power series 
$$
g(z) = \frac{dh}{dz} = -\frac{1}{2} \sum_{k\geq 1}
\frac{B_{2k}}{ (2k)!} z^{2k-1} + \sum_{k\geq 1} (2k+1)h_{2k+1} z^{2k},
$$
where $h_{2k+1}$ are the odd Taylor coefficients of the Duflo function. Hence, $g(z)$
encodes the information on the full Duflo function (and not only on its even part).

\begin{rem}
In \cite{Mas15}, Massuyeau constructed expansions $\theta$ verifying the condition $\delta_{\theta}^+=\delta^{\alg}$ using the Kontsevich integral.
Moreover, he obtained a description of an operation which is closely related (and essentially equivalent) to $\mu$.
His description works for any choice of a Drinfeld associator $\Phi$ which one needs for the construction of the Kontsevich integral, and it involves the $\Gamma$-function of $\Phi$;
see \cite{En06}.
\end{rem}

Our results are based on the following theorem:
\begin{thm}  \label{thm:intro_delta_div}
Let $F$ be an automorphism of the free Lie algebra with two generators which satisfies equation \eqref{eq:intro_KV0} and let $\theta=F^{-1} \circ \theta^{\rm exp}$ be the corresponding expansion. Then, 
\begin{equation} \label{eq:intro_key}
\delta^+_\theta(|a|) = \tdiv (\{ |a|, -\}_\theta),
\end{equation}
if $j(F^{-1}) \in |A|$ belongs to the center of the Lie algebra $(|A|, \{ -, -\}_\theta)$.
\end{thm}
In proving Theorem \ref{thm:intro_delta_div} we first show it for $\theta =\theta^{\rm exp}$ and then extend it for $F$'s solving the KV problem. In the course of the proof, we need the description of the center of the Lie bracket $\{ -, -\}_{\rm KKS}$. We give a direct algebraic proof of the fact that the center is spanned by elements of the form
$$
Z(|A|, \{ -, -\}_{\rm KKS}) = \sum_{i=0}^n \, |\mathbb{K}[[x_i]]|,
$$
where $x_0= -\sum_{i=1}^n x_i$.
Equation \eqref{eq:intro_key} shows that the cocycle property of the divergence cocycle implies the cocycle property of the cobracket $\delta^+_\theta$. However, it does not shed light on the co-Jacobi identity for $\delta^+_\theta$.

The structure of the paper is as follows. In Section \ref{sec:der}, we recall the notions of tangential and double derivations and introduce the space of cyclic words. In Section \ref{sec:div}, we discuss various versions of the divergence map and derive their multiplicative properties. In Section \ref{sec:db}, we consider double brackets (in the sense of van den Bergh) and study properties of composition maps such as \eqref{eq:intro_key}. In Section \ref{sec:GTLB}, we recall the topological definitions of the Goldman bracket, Turaev cobracket and of the maps $\kappa$ and $\mu$. In Section \ref{sec:mult}, we introduce the notion of expansions and define the transfer of structures to the Hopf algebra $A$. In Section \ref{sec:KVA}, we show that solutions of the Kashiwara-Vergne problem give rise to expansions with property $\delta^+_\theta=\delta^{\rm alg}$. In Section \ref{sec:KVL}, we prove Theorem \ref{thm:intro_KV}. Appendix A contains an algebraic proof of the description of the center of the KKS Lie bracket.

Our results were announced in \cite{short}. This paper focuses on the case of genus zero. The higher genus case is more involved, and it will be considered in the forthcoming paper \cite{next}.

\vskip 1em

{\bf Acknowledgements.}
We are grateful to G. Massuyeau
for sharing with us the results of his work. We thank B. Enriquez, S. Merkulov, P. Severa and A. Tsuchiya for fruitful discussions.
The work of A.A. and F.N. was supported in part by the grants 159581 and 165666 of the Swiss National Science Foundation, by the ERC grant MODFLAT and by the NCCR SwissMAP. Research of N.K. was supported in part by the grants JSPS KAKENHI 15H03617 and 24224002. Y.K. was supported in part by the grant JSPS KAKENHI 26800044.


\section{Derivations on free associative/Lie algebras}
\label{sec:der}

In this section, we give basic definitions of derivations and double derivations on free associative and Lie algebras.

\subsection{Tangential derivations}

Let $\K$ be a field of characteristic $0$.
For a positive integer $n$, let $A=A_n$ be the degree completion of the free associative algebra over $\K$ generated by indeterminates $x_1,\ldots,x_n$.
The algebra $A$ has a complete Hopf algebra structure whose coproduct $\Delta$, augmentation $\varepsilon$, and antipode $\iota$ are defined by
\[
\Delta(x_i):=x_i\otimes 1+1\otimes x_i, \quad \varepsilon(x_i):=0, \quad \iota(x_i):=-x_i.
\]
Here and throughout the paper, we simply use $\otimes$ for the complete tensor product $\widehat{\otimes}$.

We denote by $A_{\ge 1}=\ker(\varepsilon)$ the positive degree part of $A_n$. The primitive part
$L=L_n:=\{ a\in A \mid \Delta(a)=a\otimes 1+1\otimes a\}$
of $A$ is naturally identified with the degree completion of the free Lie algebra over $\K$ generated by $x_1,\ldots,x_n$.
Sometimes we drop $n$ from the notation and simply use $A$ and $L$.

Following \cite{vdB08} \S 2.1, we consider the following natural $A$-bimodule structures on the algebra $A\otimes A$.
The \emph{outer} bimodule structure is given by
\begin{equation}
\label{eq:out}
c_1(a\otimes b)c_2:=c_1a\otimes bc_2
=(c_1\otimes 1)(a\otimes b)(1\otimes c_2),
\end{equation}
while the \emph{inner} bimodule structure is given by
\begin{equation}
\label{eq:in}
c_1*(a\otimes b)*c_2:=ac_2\otimes c_1b
=(1\otimes c_1)(a\otimes b)(c_2\otimes 1).
\end{equation}
Here, $a,b,c_1,c_2\in A$.

Let $\der(A)$ be the Lie algebra of (continuous) derivations on the algebra $A$.
Note that an element of $\der(A)$ is completely determined by its values on  generators $x_1,\ldots,x_n$.
We set
$$
\tder(A):=A^{\oplus n},
$$
and define a map $\rho: \tder(A) \to \der(A)$ such that for
 $u=(u_1,\ldots,u_n)\in \tder(A)$ we have
$$
\rho(u)(x_i):=[x_i,u_i].
$$
Then, we can equip $\tder(A)$ with a Lie bracket given as follows: for $u=(u_1,\ldots,u_n),v=(v_1,\ldots,v_n)\in \tder(A)$, we have $[u,v]=w$, where $w=(w_1,\ldots,w_n)$ and 
$$
w_i:=\rho(u)(v_i)-\rho(v)(u_i)+[u_i,v_i].
$$
The map
$\rho: \tder(A)\to \der(A), u\mapsto \rho(u)$
becomes a Lie algebra homomorphism.

\begin{rem}  \label{rem:kernel_rho}
The kernel of the map $\rho$ is spanned by elements of the form 
$u=(h_1(x_1), \dots, h_n(x_n))$, where $h_i(z) \in \mathbb{K}[[z]]$ are formal power series in one variable.
\end{rem}

\begin{rem}
The map $\rho$ lifts to an injective Lie algebra homomorphism
$\hat\rho: \tder(A) \to A^{\oplus n}\rtimes \der(A)$ 
defined by $\hat\rho(u) := (u, \rho(u)) $. The
target is a semi-direct product of Lie algebras 
$A^{\oplus n}$ and $\der(A)$, where
$\der(A)$ acts diagonally 
on the $n$-fold direct product $A^{\oplus n}$.
\end{rem}

We define Lie algebras $\der(L)$ and $\tder(L)$ by replacing $A$ in the preceding paragraphs with the Lie algebra $L$.
In particular, $\tder(L)=L^{\oplus n}$ as a set.
In an obvious way, $\der(L)$ and $\tder(L)$ are Lie subalgebras of $\der(A)$ and $\tder(A)$, respectively.
We use the same notation $\rho: \tder(L) \to \der(L)$ for the restriction of $\rho$ to $\tder(L)$. The kernel of this map is spanned by the elements 
\begin{equation}  \label{eq:q_i}
q_i = (0, \dots, x_i, \dots, 0) \in \tder(L),    
\end{equation}
where $x_i$ is placed in the $i$th position.

\begin{rem}
Here we follow the notation in \cite{AET10}.
In \cite{AT12}, the notation is different and elements in the image of $\rho$ are called \emph{tangential} derivations.
\end{rem}

\subsection{Double derivations}

\begin{dfn}
A \emph{double derivation} on $A$ is a $\K$-linear map $\phi:A\to A\otimes A$ satisfying
\[
\phi(ab)=\phi(a)b+a\phi(b)
\]
for any $a,b\in A$.
\end{dfn}

In the definition,  we make use of the outer bimodule structure (\ref{eq:out}).
Let $\D_A=\der(A,A\otimes A)$ be the set of double derivations on $A$.
It has a natural $A$-bimodule structure given by
\[
(c_1 \phi c_2)(a)=c_1*\phi (a)*c_2
\]
for $c_1,c_2\in A$ and $\phi\in \D_{A}$.
Here we use the inner bimodule structure (\ref{eq:in}). 

As an $A$-bimodule, $\D_A$ is free of rank $n$.
In fact, the map
$\D_A\to (A\otimes A)^{\oplus n}, \phi\mapsto (\phi(x_1),\ldots,\phi(x_n))$
is an isomorphism of $A$-bimodules where the $A$-bimodule structure on $A\otimes A$ is the inner one.
There is a basis of $\D_A$ given by double derivations $\pa_i, i = 1, \dots, n$ such that
$$
\partial_i(x_j)= \delta_{i,j}( 1 \otimes 1),
$$
where $\delta_{i,j}$ is the Kronecker $\delta$-symbol. 
It is convenient to use the Sweedler convention for the action of double derivations:
$\phi(a)  = \phi'(a) \otimes \phi''(a)$.
For instance, we write $\phi(x_i) = \phi'_i \otimes \phi''_i$ which implies $\phi=\sum_i \phi''_i \partial_i \phi'_i$.
We also write $\partial_i a = \partial'_i a \otimes \partial''_i a$. Using this notation, we obtain
\begin{equation}   \label{eq:phi(a)}
\phi(a) = \sum_i \phi''_i \partial_i \phi'_i (a) = \sum_i (\partial'_i a) \phi'_i \otimes  \phi''_i (\partial''_i a).
\end{equation}

The Lie algebra of derivations $\der(A)$ naturally acts on $A \otimes A$ as follows: for $u \in \der(A), a,b \in A$ we have
$ u(a \otimes b) = u(a) \otimes b + a \otimes u(b)$.
This implies that the space of double derivations  $\D_A$ carries a natural action of $\der(A)$,
$$
u: \phi \mapsto  [u, \phi] = u \circ \phi - \phi \circ u
= (u\otimes 1 + 1\otimes u) \circ \phi - \phi \circ u.
$$
The double derivation $[u, \phi]$ is completely determined by the $n$-tuple $[u, \phi](x_j)$ for $j=1, \dots, n$. For example, for $\phi=\partial_i$, we have
\begin{equation}    \label{partial_commute}
[u, \partial_i](x_j) = \delta_{i,j} u(1 \otimes 1) - \partial_i u(x_j) = - \partial'_i u(x_j) \otimes \partial''_i u(x_j).
\end{equation}

Let $\tD_A$ denote the $A$-bi-submodule of $\D_A$ generated by
$\ad_{x_i}\pa_i=x_i\pa_i-\pa_ix_i$, $i=1,\ldots,n$.
Note that 
$$
(\ad_{x_i}\pa_i) (x_j)=\delta_{i,j}(1\otimes x_i-x_i\otimes 1).
$$
Sometimes it is convenient to have a special notation for the injective map $i: \tD_A \to \D_A$ which simply sends an element $\phi \in \tD_A$ to the same tangential double derivation  $\phi$ but viewed as an element of $\D_A$.

In what follows we will need the following Lemma:
\begin{lem}
Let $C \in A \otimes A$ and assume that $Cx_i - x_i C=0$ for some $i$. Then, $C=0$.
\end{lem}

\begin{proof}
Introduce a bi-degree corresponding to the degrees in the first and second copies of $A$ in $A \otimes A$.
Since the equation $Cx_i - x_iC=0$ is homogeneous, we can assume that $C$ is of some total degree $l$.
Then, $C=\sum_{k=0}^l C^{k, l-k}$ splits into a sum of terms with given bi-degrees. The equation $Cx_i - x_iC=0$ reads
$$
C^{k+1, l-k-1} x_i - x_i C^{k, l-k} =0.
$$
In particular, for $k=l$ the equation yields  $x_i C^{l,0}=0$ which implies $C^{l,0}=0$. For $k=l-1$ we obtain $0=C^{l,0} x_i - x_i C^{l-1,1}$ which implies $C^{l-1, 1}=0$. By induction, we get $C^{k, l-k}=0$ for all $k$ and $C=0$, as required.
\end{proof}

\begin{lem}
\label{lem:ali}
$\tD_A \subset \D_A$ is a free $A$-bimodule of rank $n$ with a basis $\{ \ad_{x_i}\pa_i\}_{i=1}^n$.
 \end{lem}

\begin{proof}
Assume that  $\sum_i C'_i (\ad_{x_i}\pa_i) C''_i=0$ for some $C_i=C'_i \otimes C''_i \in A \otimes A, i=1, \dots, n$.
By applying this double derivation to $x_j$, we obtain $(C_j)^{\circ}x_j-x_j(C_j)^{\circ}=0$, where $(a\otimes b)^{\circ}=b\otimes a$ for $a,b\in A$. Hence, $C_j=0$ for all $j$, as required.
\end{proof}

As a sample computation, we introduce an element of $\tD_A$ which will be used later.

\begin{lem}
\label{lem:phi_0}
Let $\phi_0:= \sum_i \ad_{x_i}\pa_i \in \tD_A$.
Then, $\phi_0(a)=1\otimes a-a\otimes 1$ for all $a\in A$.
\end{lem}

\begin{proof}
Both $\phi_0:= \sum_i \ad_{x_i}\pa_i$ and the map $a \mapsto 1\otimes a - a \otimes 1$ are double derivations on $A$.
Since $\phi_0(x_i)=1 \otimes x_i - x_i \otimes 1$ they agree on generators of $A$, and therefore they coincide.
\end{proof}

Through the map $\rho$, the Lie algebra $\tder(A)$ acts on $\D_A$.
\begin{lem} \label{lem:tdertD}
The action of $\tder(A)$ on $\D_A$ preserves $\tD_A$.
\end{lem}

\begin{proof}
Let $u\in \tder(A)$ and $\phi=\sum_i c_i'(\ad_{x_i}\pa_i) c_i''\in \tD_A$. Then, by a straightforward computation, we see that
$(u\cdot \phi)(x_i)=[\rho(u),\phi](x_i) =u\cdot \phi(x_i)-\phi(u(x_i))$ is equal to
the value at $x_i$ of the following double derivation
\[
u(c'_i)(\ad_{x_i}\pa_i)c''_i+c'_i(\ad_{x_i}\pa_i)u(c''_i)
+\phi''(u_i)(\ad_{x_i}\pa_i)\phi'(u_i)
+c_i'(\ad_{x_i}\pa_i)u_ic_i''-c_i' u_i(\ad_{x_i}\pa_i)c_i''.
\]
This proves the assertion.
\end{proof}

\subsection{Cyclic words and extra parameters}

For a (topological) $A$-bimodule $N$, let $[A,N]$ be (the closure of) the $\K$-linear subspace of $A$ spanned by elements of the form $an-na$, where $a\in A$ and $n\in N$. 
Set $|N|:=N/[A,N]$ and let $|\ |:N\to |N|$ be the natural projection. 
Since $A$ is naturally an $A$-bimodule, we can apply the construction described above to obtain
$$
|A|:=A/[A,A].
$$
Note that $|A|$ is also isomorphic to the $\mathbb{K}$-span of cyclic words in letters $x_1, \dots, x_n$. Examples of such words are $|x_1|$, 
$|x_1x_2|=|x_2x_1|$ {\em etc}.

The space of double derivations $\D_A$ is also an $A$-bimodule. 

\begin{lem}
The map
$$
\phi \mapsto |\phi|: a \mapsto \phi'(a) \phi''(a)
$$
establishes an isomorphism $|\D_A| \cong \der(A)$.
\end{lem}

\begin{proof}
First, observe that
$$
|b\phi - \phi b|: a \mapsto \phi'(a) (b\phi''(a)) - (\phi'(a)b) \phi''(a) =0
$$
for any $\phi\in \D_A$ and $b\in A$.
Therefore, the map $\phi \mapsto |\phi|$ descends to a map $|\D_A| \to \der(A)$. This map is surjective since 
$\sum_i |a_i \partial_i|: x_i \mapsto a_i$
for arbitrary $a_i \in A$, and hence we obtain all derivations of $A$ in that way. Finally, this map is injective since $\D_A$ is a free $A$-bimodule of rank $n$ with generators $\partial_i$ and $\{ \sum_i a_i \partial_i; \,\, a_i \in A\} \subset \D_A$
is a section of the canonical projection $\D_A \to |\D_A|$.
\end{proof}

\begin{exple}  \label{example:inner}
Let $a \in A$ and consider the derivation $u=|a \phi_0|=|a \sum_i {\rm ad}_{x_i} \partial_i|$. By computing it on generators, we obtain
$$
u(x_j)= [a, x_j].
$$
Hence, it is an inner derivation with generator $a$.
\end{exple}

For the $A$-bimodule $\tD_A$, we define the map
$$
\psi =\sum_i c'_i(\ad_{x_i}\pa_i)c''_i \mapsto |\psi|= - (c''_1c'_1,\ldots, c''_nc'_n) \in \tder(A)
$$
which defines an isomorphism $|\tD_A| \to \tder(A)$. It is easy to check that $\rho(|\psi|) = |i(\psi)|$ for all $\psi \in \tD_A$. 

\begin{lem} \label{lem:tD_A_equiv}
The map $\tD_A \to \tder(A), \psi \mapsto |\psi|$ is $\tder(A)$-equivariant.
\end{lem}

\begin{proof}
Let $u\in \tder(A)$ and $\phi=\sum_i c'_i (\ad_{x_i}\pa_i) c''_i \in \tD_A$. The coefficient of $\ad_{x_i}\pa_i$ in $u\cdot \phi \in \tD_A$ is computed in the proof of Lemma \ref{lem:tdertD}.
Then the $i$th component of $|u\cdot \phi|\in \tder(A)$ is given by
\begin{align*}
& -\left( c''_iu(c'_i)+u(c''_i)c'_i+\phi'(u_i)\phi''(u_i)+u_ic''_ic'_i -c''_ic'_iu_i \right) \\
=& u(-c''_i c'_i)-|\phi| (u_i)+[u_i, -c''_i c'_i] \\
=& u(|\phi|_i)-|\phi|(u_i)+[u_i,|\phi|_i] \\
=& [u,|\phi|]_i.
\end{align*}
This proves that $|u\cdot \phi|=[u,|\phi|]$, as required.
\end{proof}

Let $A^T$ be the degree completion of the free associative algebra over $\K$ generated by $x_1,\ldots,x_n$, and $T$.
The algebra $A^T$ is isomorphic to $A_{n+1}$, but in what follows the extra generator $T$ will play a different role from the other generators. Define an injective map $\D_A \to \der(A^T)$ defined by formula
$$
\phi \mapsto \phi^T: a \mapsto \phi'(a) T \phi''(a)
$$
for $a \in A$ and $\phi^T(T) =0$. The image of this map consists of derivations of $A^T$ which are linear in $T$ and which vanish on the generator $T$. The map $\phi \mapsto \phi^T$ commutes with the natural actions of $\der(A)$ on $\D_A$ and on $\der(A^T)$.
By composing with the evaluation map ${\rm ev}_1$ (which puts the extra generator $T$ equal 1), we obtain a derivation of $A$:
$a \mapsto ev_1(\phi^T(a))$ which coincides with $|\phi|$.

\begin{rem}  \label{rem:double_acts}
Let $\phi \in \D_A$ be a double derivation. The derivation $\phi^T \in \der(A^T)$ descends to an endomorphism of $|A^T|$ and it maps $|A| \subset |A^T|$ to cyclic words linear in $T$. But the space of cyclic words linear in $T$ is isomorphic to $A$ via the map $a \mapsto |Ta|$. Hence, the double derivation $\phi$ defines a map $\phi: |A| \to A$ by formula
$$
\phi: |a| \mapsto \phi^T(|a|) = |\phi'(a)T\phi''(a)| \mapsto \phi''(a)\phi'(a).
$$
\end{rem}

Similarly, we define a map $\tD_A \to \tder(A^T)$ by formula
$$
\psi=\sum_i c'_i(\ad_{x_i}\pa_i)c''_i \mapsto \psi^T:= - (c''_1Tc'_1,\ldots, c''_nTc'_n,0).
$$
As before, it is equivariant under the action of $\tder(A)$ on $\tD_A$ and under the adjoint action of $\tder(A)$ on itself.
Again, it is easy to check that $\rho(\psi^T) = i(\psi)^T$.

%
%

\subsection{Integration to groups}

Let $\mathfrak{g}$ be a non-negatively graded Lie algebra and assume that the degree zero part $\mathfrak{g}_0$ is contained in the center of $\mathfrak{g}$. Then, one can use the Baker-Campbell-Hausdorff formula to define a group law on $\mathfrak{g}$. We denote the corresponding group by $\mathcal{G}$.  It is convenient to use the multiplicative notation for the group elements: to $u\in \mathfrak{g}$ corresponds the group element $e^u \in \mathcal{G}$ and $e^u e^v= e^{u * v}$, where
$$
u*v = {\rm bch}(u,v)=\log(e^u e^v)= u+ v + \frac{1}{2} [u, v] + \dots 
$$

Lie algebras $\der^+(A)$ (derivations of $A$ of positive degree), $L$ and $\tder(L)$ are positively graded pro-nilpotent Lie algebras. Lie algebras $A$ (with Lie bracket the commutator) and $\tder(A)$ are non-negatively graded, and their zero degree parts are contained in the center. Hence, the BCH formula defines a group law on all of them.  We denote the groups integrating $A$ and $L$ by $\exp(A)$ and $\exp(L)$, respectively. For the integration of $\tder(A)$ and $\tder(L)$ we use the notation $\taut(A)$ and $\taut(L)$.

As a set, the group $\taut(A)$ is isomorphic to 
$$
\taut(A) \cong \left( \exp(A) \right)^{\times n}.
$$
Elements of $\taut(A)$ are $n$-tuples $F=(F_1, \dots, F_n)$, where $F_i \in \exp(A)$.
The map $\rho: \tder(A) \to \der(A)$ lifts to a map (denoted by the same letter) $\rho: \taut(A) \to {\rm Aut}(A)$ described in the following way.
First of all, note that any element $F=\exp(f)\in \exp(A)$ with $f\in A$ decomposes as $F=F_0F_{\ge 1}$, where $f=f_0+f_{\ge 1}$ is the decomposition into the degree $0$ part $f_0\in \K$ and the positive degree part $f_{\ge 1}\in A_{\ge 1}$, $F_0=\exp(f_0)$ and $F_{\ge 1}=\exp(f_{\ge 1})$.
Then, $F_{\ge 1}$ makes sense as an invertible element of $A$, by the exponential series $\exp(f_{\ge 1})=\sum_{k=0}^{\infty}(1/k!) {f_{\ge 1}}^k$.
With this notation, the whole group $\exp(A)$ acts on $A$ by inner automorphisms:
\[
F.a:=(F_{\ge 1})^{-1} a F_{\ge 1}, \quad \text{for $F\in \exp(A)$, $a\in A$}.
\]
We will use the shorthand notation $F^{-1}aF$ for $(F_{\ge 1})^{-1} a F_{\ge 1}$.
Now, the map $\rho: \taut(A) \to {\rm Aut}(A)$ is defined by
$$
\rho(F): x_i \mapsto F_i^{-1} x_i F_i.
$$
The group law on $\taut(A)$ is given by
$$
(F \cdot G)_i= F_i (\rho(F) G_i),
$$
where $F \cdot G$ stands for the product in $\taut(A)$ and the product $F_i (\rho(F) G_i)$ takes place in $\exp(A)$. A similar description applies to the group $\taut(L)$.
Note that if $F\in \taut(L)$, then $F_{\ge 1}=F$.

All Lie algebra actions described in the previous section integrate to group actions. For instance, the adjoint action of $\tder(A)$ on itself integrates to the adjoint action ${\rm Ad}$ of $\taut(A)$ on $\tder(A)$, and 
the natural action of $\tder(A)$ on $\tD_A$ (Lemma \ref{lem:tdertD}) integrates to a group action of $\taut(A)$ on $\tD_A$.
By Lemma \ref{lem:tD_A_equiv}, the map $\psi \mapsto |\psi|$ from $\tD_A$ to $\tder(A)$ is $\taut(A)$-equivariant: for any $F \in \taut(A)$, we have
$$
|F.\psi|= {\rm Ad}_F \, |\psi|.
$$

\section{Divergence maps}
\label{sec:div}

In this section, we describe various versions of the non-commutative divergence map and describe their properties.

\subsection{Double divergence maps}

Consider  $|A|:=A/[A,A]$ and let ${\bf 1}\in |A|$ denote the image of the unit $1\in A$.
The natural action of $\der(A)$ on $A$ descends to an action on $|A|$.
By using the map $\rho$, we can view $A$ and $|A|$ as left $\tder(A)$-modules.
In the following proposition, we introduce the \emph{divergence map} $\ddiv: \der(A) \to |A| \otimes |A|$.
By abuse of notation, we use the letter $|\ |$ for the map $|\ |^{\otimes 2}: A\otimes A\to |A|\otimes |A|$.

\begin{prop}
\label{prop:Div}
The map
$\ddiv: \der(A)\to |A|\otimes |A|$
defined by formula
$$
\ddiv(u):=\sum_i |\pa_i(u(x_i))|
$$
for $u\in \der(A)$ is a $1$-cocycle on the Lie algebra $\der(A)$. That is, for all $u,v \in \der(A)$, we have
$$
\ddiv([u,v]) = u \cdot \ddiv(v) - v \cdot \ddiv(u).
$$
\end{prop}

\begin{proof}
Consider a map $Z_i: A \to A \otimes A$ defined as follows:
$$
Z_i:= \partial_i \circ [u,v] - u \circ \partial_i \circ v + v \circ \partial_i \circ u = [v, \partial_i] \circ u - [u, \partial_i] \circ v.
$$
When applied to the generator $x_i$, it yields
$$
\begin{array}{lll}
Z_i(x_i) & = & [v, \partial_i](u(x_i)) - [u, \partial_i](v(x_i)) \\
& = & \sum_k \left( - (\partial'_k u(x_i))(\partial'_i v(x_k)) \otimes (\partial''_i v(x_k)) (\partial''_k u(x_i)) \right. \\
&  & \hspace{2em} +
\left. (\partial'_k v(x_i))(\partial'_i u(x_k)) \otimes (\partial''_i u(x_k)) (\partial''_k v(x_i)) \right) ,
\end{array}
$$
where we have used equation \eqref{partial_commute}.
Note that
$$
\ddiv([u,v]) - u \cdot \ddiv(v) + v \cdot \ddiv(u) = \sum_i |Z_i(x_i) |,
$$
and combining with the previous formula we obtain
$$
\begin{array}{lll}
\sum_i |Z_i(x_i) | & = &  \sum_{i,k} \left( - |(\partial'_k u(x_i))(\partial'_i v(x_k))| \otimes |(\partial''_i v(x_k)) (\partial''_k u(x_i)) |\right. \\
& & \hspace{3em} +
\left. |(\partial'_k v(x_i))(\partial'_i u(x_k))| \otimes |(\partial''_i u(x_k)) (\partial''_k v(x_i))| \right)  = 0.
\end{array}
$$
Here the two terms cancel each other after using the cyclic property of $| \cdot |$ and renaming $i$ and $k$.
\end{proof}

The Lie algebra $\tder(A)$, which is related to $\der(A)$ by the map $\rho: \tder(A)\to \der(A)$, carries additional 1-cocycles as explained in the following lemma.

\begin{lem}
\label{lem:c_i}
For $1\le i\le n$, the map $c_i: \tder(A)\to |A|$ defined by formula
$c_i(u):=|u_i|$ for $u=(u_1,\ldots,u_n)\in \tder(A)$ is a $1$-cocycle on the Lie algebra $\tder(A)$.
\end{lem}

\begin{proof}
Let $u=(u_1,\ldots,u_n),v=(v_1,\ldots,v_n)\in \tder(A)$.
The $i$th component of $[u,v]$ is
$w_i=\rho(u)v_i-\rho(v)u_i+[u_i,v_i]$.
Since $|[u_i,v_i]|=0$, we obtain
$$
c_i([u,v])=|w_i|=|\rho(u) v_i| - |\rho(v) u_i| =
u\cdot |v_i|-v\cdot |u_i|=u\cdot c_i(v)-v\cdot c_i(u).
$$
\end{proof}

We define the \emph{tangential divergence} cocycle $\tdiv: \tder(A) \to |A| \otimes |A|$ as a linear combination of the divergence cocycle and of the cocycles $c_i$.

\begin{dfn}
For $u \in \tder(A)$, set
$$ 
\tdiv(u) = \ddiv(\rho(u)) + \sum_i (c_i(u) \otimes {\bf 1} - {\bf 1} \otimes c_i(u)).
$$
\end{dfn}
Since $\rho$ is a Lie algebra homomorphism and $\ddiv$ and $c_i$'s are 1-cocycles, it is obvious that $\tdiv$ is a 1-cocylce as well. In the following lemma we give a more explicit formula for $\tdiv$.

\begin{lem} 
For $u = (u_1, \dots, u_n) \in \tder(A)$, we have
$$
\tdiv(u) = \sum_i |x_i (\partial_i u_i) - (\partial_i u_i) x_i|.
$$
\end{lem}

\begin{proof}
We compute
$$
\begin{array}{lll}
\tdiv(u) & = & \sum_i | \partial_i [x_i, u_i]| + \sum_i ( |u_i| \otimes {\bf 1} - {\bf 1} \otimes |u_i|) \\
& = & \sum_i |1 \otimes u_i + x_i (\partial_i u_i) - (\partial_i u_i) x_i - u_i \otimes 1| + \sum_i ( |u_i| \otimes {\bf 1} - {\bf 1} \otimes |u_i|) \\
& = & \sum_i |x_i (\partial_i u_i) - (\partial_i u_i) x_i|.
\end{array}
$$
\end{proof}

\subsection{Relation to the ordinary divergence}

Recall the \emph{divergence map} introduced in \cite{AT12} \S 3.3.
Any element $a\in A$ can be uniquely written as $a=a^0+\sum_i a^i x_i$, where $a^0\in \K$ and $a^i\in A$. 
One defines the map $\div: \tder(L)\to |A|$ by formula 
$$
\div(u):=\sum_i |x_i(u_i)^i |,
$$
where $u=(u_1,\ldots,u_n)\in \tder(L)$. The map $\div$  is a 1-cocycle on the Lie algebra $\tder(L)$
with values in $|A|$.

Consider an algebra homomorphism from $A$ to $A\otimes A^{\op}$ defined by formula
$$
\tde:=(1\otimes \iota)\Delta: A\to A\otimes A.
$$
It induces a map from $|A|$ to $|A|\otimes |A|$, for which we use the same letter $\tde$.
In the following proposition, we show that $\tde$ intertwines the divergence maps $\div$ and $\tdiv$.

\begin{prop}
\label{prop:dext}
The following diagram is commutative:
\[
\begin{CD}
\tder(L) @> \div >> |A| \\
@VVV @VV \tde V \\
\tder(A) @> \tdiv >> |A|\otimes |A|
\end{CD}
\]
\end{prop}

\begin{lem}
\label{lem:pay}
For any $a\in L$ and $1\le i\le n$, we have $\tde(a^i)=\pa_i a$.
\end{lem}

\begin{proof}
Both maps $L \to A \otimes A$, $a \mapsto \pa_i a$ and $a \mapsto \tde(a^i)$ have the property
$$
\phi([a,b]) = \phi(a) b + a \phi(b) - b \phi(a) - \phi(b) a.
$$
Since $\pa_i(x_j) = \tde((x_j)^i) = \delta_{i,j} (1\otimes 1)$, we conclude that $\tde(a^i)=\pa_i a$ for all $a \in L$, as required.
%
\end{proof}

\begin{proof}[Proof of Proposition \ref{prop:dext}]
Note that we have the following expression for $\tdiv(u)$:
\begin{equation}
\label{eq:tdivA}
\tdiv(u)=\sum_i | \tde(x_i) (\pa_i u_i) |
\end{equation}
for any $u\in \tder(A)$, where the product of $\tde(x_i)=x_i\otimes 1-1\otimes x_i$ and $\pa_i u_i$ is taken in $A\otimes A^{\op}$.
Now let $u\in \tder(L)$.
Then, using equation (\ref{eq:tdivA}) and Lemma \ref{lem:pay}, we compute
\[
\tde(\div(u))=\tde\left(\sum_i |x_i(u_i)^i| \right)
=\sum_i |\tde(x_i)\tde((u_i)^i) |
=\sum_i |\tde(x_i)(\pa_i u_i) |
=\tdiv(u).
\]
\end{proof}

The following lemma will be used later.

\begin{lem}
\label{lem:tdeeq}
The maps $\tde:A\to A\otimes A$ and $\tde:|A|\to |A|\otimes |A|$ are injective.
They are maps of $\der(L)$-modules.
\end{lem}

\begin{proof}
The injectivity follows from $(1\otimes \varepsilon)\tde(a)=(1\otimes \varepsilon)\Delta(a)=a$ for any $a\in A$.
Since the coproduct $\Delta$ and the antipode $\iota$ are morphisms of $\der(L)$-modules, so is the map $\tde$.
\end{proof}

\subsection{Divergence maps with extra parameters}

Next we introduce a certain refinement of $\tdiv$.
Recall that  $A^T$ is the degree completion of the free associative algebra over $\K$ generated by $x_1,\ldots,x_n$, and $T$.
We have an embedding
\[
j_T: (|A|\otimes A)\oplus (A \otimes |A|) \to |A^T|\otimes |A^T|, \quad
(|a|\otimes b,c\otimes |d|)\mapsto |a|\otimes |Tb|+|Tc|\otimes |d|,
\]
through which we identify the source and the image.
Let $p_1:\mathrm{im}(j_T) \to |A|\otimes A$ and $p_2:\mathrm{im}(j_T)\to A \otimes |A|$ be projections on the  first and the second factors in the tensor product, respectively.
The map $\tD_A\to \tder(A^T), \phi \mapsto \phi^T$ has the property that
$\tdiv(\phi^T) \in \mathrm{im}(j_T)$ 
since every component of $\phi^T$ (with the exception of the last one which vanishes) is linear in $T$.

\begin{dfn}
\label{dfn:tdiv^T}
Define the maps $\tdiv^T: \tD_A\to |A|\otimes A$ and
$\underline{\tdiv}^T: \tD_A\to A\otimes |A|$
by
$\tdiv^T(\phi):=p_1(\tdiv(\phi^T) )$ and
$\underline{\tdiv}^T(\phi):=p_2(\tdiv(\phi^T))$
for $\phi\in \tD_A$.
\end{dfn}

Let $\phi=\sum_i c'_i(\ad_{x_i} \pa_i) c''_i\in \tD_A$.
By (\ref{eq:tdivA}),
$\tdiv(\phi^T)=-\sum_i |\tde(x_i) \partial_i(c''_iTc'_i)|$, and the terms where $T$ appears to the right of the symbol $\otimes$ are
$$
-\sum_i |\tde(x_i)((\pa_ic''_i)Tc'_i)|=-\sum_i |(x_i\otimes 1)(\pa_i c''_i)(1\otimes Tc'_i)-(\pa_i c''_i)(1\otimes Tc'_i)(1\otimes x_i) |.
$$
Therefore, we obtain
\begin{equation}
\label{eq:tdiv^T}
\tdiv^T(\phi)=(|\ |\otimes 1) \left( \sum_i (1\otimes c'_ix_i-x_i\otimes c'_i)(\pa_i c''_i) \right).
\end{equation}
In a similar way, we obtain
\begin{equation}
\label{eq:tdiv^T2}
\underline{\tdiv}^T(\phi)=
(1\otimes |\ |)\left( \sum_i (\pa_i c'_i) (c''_i\otimes x_i-x_ic''_i \otimes 1) \right).
\end{equation}

\begin{prop}
\label{prop:aphi}
For any $a\in A$ and $\phi\in \tD_A$,
\begin{align*}
\tdiv^T(a\phi) =& (1\otimes a)\tdiv^T(\phi), \\
\tdiv^T(\phi a)=& \tdiv^T(\phi)(1\otimes a)+(|\ |\otimes 1)\phi(a), \\
\underline{\tdiv}^T(a\phi)=& (a\otimes 1)\underline{\tdiv}^T(\phi)+(1\otimes |\ |)\phi(a), \\
\underline{\tdiv}^T(\phi a)=& \underline{\tdiv}^T(\phi)(a\otimes 1).
\end{align*}
\end{prop}

\begin{proof}
The first equation is an immediate consequence of  \eqref{eq:tdiv^T}.
For the second one, we compute,
\begin{align*}
\tdiv^T(\phi a)= &(|\ |\otimes 1)\left( \sum_i (1\otimes c'_ix_i-x_i\otimes c'_i) \pa_i (c''_i a) \right) \\
=& (|\ |\otimes 1)\left( \sum_i (1\otimes c'_ix_i-x_i\otimes c'_i)(\pa_ic''_i)a \right) \\
&+ (|\ |\otimes 1)\left( \sum_i (1\otimes c'_ix_i-x_i\otimes c'_i)(c''_i\otimes 1)\pa_i a \right) \\
=& \tdiv^T(\phi)(1\otimes a)+(|\ |\otimes 1)\left( \sum_i \phi(x_i) \pa_i a \right).
\end{align*}
Since $(|\ |\otimes 1)\phi(a)=(|\ |\otimes 1)( \sum_i \phi(x_i) (\pa_i a))$, the second equation follows.
The formulas for $\underline{\tdiv}^T$ can be obtained similarly
(we use $(1\otimes |\ |)\phi(a)=(1\otimes |\ |)(\sum_i (\pa_i a)\phi(x_i))$).
\end{proof}

\begin{prop}
\label{prop:rel_div}
For any $\phi\in \tD_A$, we have
\[
\tdiv(|\phi|)=(1\otimes |\ |)\tdiv^T(\phi)+
(|\ |\otimes 1)\underline{\tdiv}^T(\phi).
\]
\end{prop}

\begin{proof}
By (\ref{eq:tdiv^T}) and (\ref{eq:tdiv^T2}), we compute
\begin{align*}
& (1\otimes |\ |)\tdiv^T(\phi)+(|\ |\otimes 1)\underline{\tdiv}^T(\phi) \\
=& (|\ |\otimes |\ |)\left( \sum_i (1\otimes x_i-x_i\otimes 1)(\pa_i c''_i)(1\otimes c'_i)
+\sum_i (1\otimes x_i-x_i\otimes 1) (c''_i \otimes 1) (\pa_i c'_i)
\right) \\
=& (|\ |\otimes |\ |) \left( \sum_i (x_i\otimes 1-1\otimes x_i) 
\pa_i( -c''_ic'_i) \right) \\
=& \tdiv(|\phi|).
\end{align*}
\end{proof}

\subsection{Integration of divergence cocycles}

Let $\mathfrak{g}$ be non-negatively graded Lie algebra with the zero degree part $\mathfrak{g}_0$ contained in the center of $\mathfrak{g}$. Let 
$\mathcal{G}$ be the group integrating $\mathfrak{g}$ (using the BCH series) and  $M$ a graded module over $\mathfrak{g}$. Assume that the action of $\mathfrak{g}_0$ on $M$ is trivial. Then, the Lie algebra action of $\mathfrak{g}$ integrates to a group action of $\mathcal{G}$ on $M$.
Let $c: \mathfrak{g} \to M$ be a Lie algebra 1-cocycle. Under the conditions listed above, it integrates to a group 1-cocycle $C: \mathcal{G} \to M$ which satisfies the following defining properties: the first one is the group cocycle condition
\begin{equation}
\label{eq:jco}
C(fg) = C(f) + f \cdot C(g),
\end{equation}
for all $f,g \in \mathcal{G}$. 
The second one is the relation between $c$ and $C$:
\begin{equation}
\label{eq:j'=div}
\frac{d}{dt} C(\exp(tu))|_{t=0}=c(u),
\end{equation}
for all $u \in \mathfrak{g}$. These two conditions imply 
\begin{equation}
\label{eq:jfdiv}
C(\exp(u))=\frac{e^u-1}{u}\cdot c(u)
\end{equation}
for $u \in \mathfrak{g}$ and $\exp(u) \in \mathcal{G}$.

The $1$-cocycles  $\ddiv: \der^+(A) \to |A| \otimes |A|$ (the restriction of the 1-cocycle $\ddiv$ to derivations of positive degree), $\tdiv: \tder(A) \to |A| \otimes |A|$ and $\div: \tder(L) \to |A|$ 
integrate to group 1-cocycles ${\rm J}: \exp(\der^+(A)) \to |A| \otimes |A|, \tJ: \taut(A) \to |A| \otimes |A|$ and  $j: \taut(L) \to |A|$.
 The statement of Proposition \ref{prop:dext} and Lemma \ref{lem:tdeeq} gives rise to the following relation between group cocycles:
\begin{equation}
\label{eq:tdej}
\tde(j(F))=\tJ (F)
\end{equation}
for all $F \in \taut(L)$.
The following lemma establishes a property of Lie algebra and group cocycles which will be important later in the text:
\begin{lem}
Let $M$ be a graded $\mathfrak{g}$-module, $c: \mathfrak{g} \to M$ a Lie algebra 1-cocycle and $C: \mathcal{G} \to M$ the group 1-cocycle integrating $c$. Then, for all $u \in \mathfrak{g}, f \in \mathcal{G}$ we have
\begin{equation}   \label{eq:cocycle_f}
c({\rm Ad}_f u) = f \cdot \left(c(u) + u \cdot C(f^{-1}) \right).
\end{equation}
\end{lem}

\begin{proof}
Consider the expression $C(f e^{tu} f^{-1})$ and compute the $t$-derivative at $t=0$ in two ways. On the one hand,
$$
\frac{d}{dt} \, C(f e^{tu} f^{-1})|_{t=0} = \frac{d}{dt} \, C(e^{t {\rm Ad}_f u})|_{t=0} = c({\rm Ad}_f u).
$$
On the other hand,
$$
\begin{array}{lll}
\frac{d}{dt} \, C(f e^{tu} f^{-1})|_{t=0} & =  & \frac{d}{dt} \,\left( C(f) + f\cdot C(e^{tu}) + (f e^{tu}) \cdot C(f^{-1})\right)|_{t=0} \\
& = & f \cdot c(u) + f\cdot (u \cdot C(f^{-1}) ) \\
& = & f \cdot \left( c(u) + u \cdot C(f^{-1}) \right).
\end{array}
$$

\end{proof}
Among other things, equation \eqref{eq:cocycle_f} implies that $c({\rm Ad}_f u) = f \cdot c(u)$ if and only if $u \cdot C(f^{-1})=0$.

By applying equation \eqref{eq:cocycle_f} to the cocycle $\tdiv$, we obtain the following equation which will be of importance in the next sections:
$$
\tdiv({\rm Ad}_F u) = F \cdot (\tdiv(u) + u \cdot \tJ(F^{-1}) ),
$$
where $u \in \tder(A), F \in \taut(A)$. It is convenient to introduce a ``pull-back'' of the $\tdiv$ cocycle by the automorphism $F$: $F^* \tdiv(u)= F^{-1}\cdot \tdiv({\rm Ad}_F u)$. It is again a 1-cocycle on $\tder(A)$, and it satisfies the formula
\begin{equation} \label{eq:F*tdiv}
F^* \tdiv(u) = \tdiv(u) + u\cdot \tJ(F^{-1}).
\end{equation}

There is a similar transformation property for the map $\tdiv^T: \tD_A \to |A| \otimes A$. We compute
\begin{equation} \label{eq:F*tdivT}
\begin{array}{lll}
F^* \tdiv^T(\phi) & = & (F^{-1}\otimes F^{-1}) \cdot \tdiv^T((F\otimes F) \circ \phi \circ F^{-1})\\
& = &  (F^{-1} \otimes F^{-1}) \cdot p_1(\tdiv({\rm Ad}_F \phi^T)) \\
&= & p_1(\tdiv(\phi^T)) + p_1 ( \phi^T \cdot \tJ(F^{-1})) \\
& = & \tdiv^T(\phi) + (1 \otimes \phi) \cdot \tJ(F^{-1}).
\end{array}
\end{equation}
Here $F$ is uniquely extended to a tangential automorphism of $A^T$ preserving $T$, and in the last line we use the map $\phi: |A| \to A$ defined in Remark \ref{rem:double_acts}.

\section{Double brackets and auxiliary operations}
\label{sec:db}

In this section, we introduce double brackets (in the sense of van den Bergh) and discuss various operations which can be built from double brackets and divergence maps.

\subsection{Definition and basic properties}

\begin{dfn}
\label{dfn:db}
A \emph{double bracket} on $A$ is a $\K$-linear map $\Pi:A\otimes A\mapsto A\otimes A, a\otimes b\mapsto \Pi(a,b)$
such that for any $a,b,c\in A$,
\begin{align}
\Pi(a,bc)=& \Pi(a,b)c+b\Pi(a,c), \label{eq:a,bc} \\
\Pi(ab,c)=& \Pi(a,c)*b+a*\Pi(b,c) \label{eq:ab,c}.
\end{align}
\end{dfn}

One often uses the notation
$$
\{ a, b\}_\Pi = \Pi(a,b) = \Pi(a,b)' \otimes \Pi(a,b)''.
$$

\begin{rem}
We use the terminology ``double bracket'' in a wider sense than \cite{vdB08}, where a double bracket is defined to be a map
satisfying (\ref{eq:a,bc}) and the skew-symmetry condition:
for any $a,b\in A$,
\begin{equation}
\label{eq:sks}
\Pi(b,a)=-\Pi(a,b)^{\circ}.
\end{equation}
Here, $(c\otimes d)^{\circ}=d\otimes c$ for $c,d\in A$.
As is easily seen, (\ref{eq:a,bc}) and (\ref{eq:sks}) imply (\ref{eq:ab,c}).
\end{rem}

For a given double bracket, there are several auxiliary operations.

\begin{lem}
Let $\Pi: A \otimes A \to A \otimes A$ be a double bracket. Then, for every $a \in A$ the map 
$$
\{ a, - \}_\Pi: b \mapsto \Pi(a,b)' \otimes \Pi(a,b)''
$$
is a double derivation on $A$.
\end{lem}

\begin{proof}
Indeed, by  equation (\ref{eq:a,bc}) we have
$$
\{ a, bc\}_\Pi= \Pi(a,b) \, c+ b \, \Pi(a,c)  = \{a, b\}_\Pi \, c + b \, \{a, c\}_\Pi,
$$
as required.
\end{proof}

By applying the map $| \ |: \D_A \to \der(A)$, we obtain a derivation on $A$ of the form
$$
| \{ a, - \}_\Pi|: b  \mapsto \Pi(a,b)'\Pi(a,b)''.
$$

\begin{lem}
The map $A \to \der(A)$ given by $a \mapsto |\{ a, -\}_\Pi|$ descends to $|A|$.
\end{lem}

\begin{proof}
We have,
$$
|\{ ac, -\}_\Pi|: b \mapsto \Pi(a,b)' c \Pi(a,b)'' + \Pi(c,b)' a \Pi(c,b)''.
$$
The right hand side is symmetric under exchange of $a$ and $c$. Hence, it vanishes on commutators, as required.
\end{proof}

Motivated by this lemma, we often use the notation
$$
|\{ a, -\}_\Pi| =  \{ |a|, -\}_\Pi.
$$
By equation (\ref{eq:a,bc}), the double bracket induces a map $\{-,-\}_{\Pi}: |A| \otimes |A|\to |A|$ defined by formula
\begin{equation}   \label{eq:bracket_on_|A|}
|a|\otimes |b|\mapsto |\{|a|,b\}_{\Pi}| =: \{|a|, |b|\}_\Pi.
\end{equation}
Furthermore, by Remark \ref{rem:double_acts}, we obtain a map
$\{-,-\}_\Pi : A\otimes |A|\to A$ defined by
\begin{equation} \label{eq:A|A|toA}
a\otimes |b|\mapsto \{a,-\}_{\Pi}(|b|)=
\Pi(a,b)''\Pi(a,b)'=: \{a,|b|\}_\Pi.
\end{equation}
All the three operations above are called the \emph{bracket} associated with $\Pi$.

\begin{lem} \label{lem:if_skew}
Let $\Pi$ be a skew-symmetric double bracket. Then, for any $a,b\in A$,
\[
\{a,|b|\}_\Pi=-\{|b|,a\}_\Pi.
\]
\end{lem}

\begin{proof}
Since $\Pi$ is skew-symmetric,
$\Pi(b,a)=-\Pi(a,b)^{\circ}=-\Pi(a,b)''\otimes \Pi(a,b)'$.
Then
\[
\{ |b|,a\}_\Pi=\Pi(b,a)'\Pi(b,a)''=-\Pi(a,b)''\Pi(a,b)'=-\{a,|b| \}_\Pi.
\]
\end{proof}

In order to give concrete examples of double brackets we consider the space $\D_A \otimes_A \D_A$ and its $A$-bimodule structure given by
$c_1 (\phi_1\otimes_A \phi_2) c_2:=(c_1\phi_1)\otimes_A (\phi_2c_2)$
for $c_1,c_2\in A$ and $\phi_1,\phi_2\in \D_A$.
Set
\[
(DA)_2:=| \D_A \otimes_A \D_A |.
\]

For $\phi_1,\phi_2\in \D_{A}$, we define a map $\Pi_{\phi_1\otimes \phi_2}:A\otimes A\to A\otimes A$ by
\begin{equation} \label{eq:phi1phi2}
\Pi_{\phi_1\otimes \phi_2}(a,b):=\phi'_2(b)\phi''_1(a)\otimes \phi'_1(a)\phi''_2(b),
\end{equation}
where we write $\phi_1(a)=\phi'_1(a)\otimes \phi''_1(a)$ and
$\phi_2(b)=\phi'_2(b)\otimes \phi''_2(b)$.
In other words,
\begin{equation}
\label{eq:ph1ph2}
\Pi_{\phi_1\otimes \phi_2}(a,-)=\phi'_1(a)\phi_2\phi''_1(a)\in \D_A
\end{equation}
for any $a\in A$.
One can check that $\Pi_{\phi_1\otimes \phi_2}$ is a double bracket on $A$.
In \cite{vdB08} \S 4, it is proved that this assignment induces a $\K$-linear isomorphism
$$
(DA)_2 \xrightarrow{\cong} \{ \text{double brackets on $A$} \}.
$$
It is convenient to describe double brackets on $A$ through this isomorphism.

\subsection{Tangential double brackets}

In this section, we consider a class of double brackets of  particular interest.

\begin{dfn}
A double bracket $\Pi$ on $A$ is called \emph{tangential} if
$\{ a,-\}_\Pi \in \tD_A$ for any $a\in A$.
\end{dfn}

Notice that $\Pi$ is tangential if and only if $\{x_i,-\}_\Pi \in \tD_A$ for every $i$. Recall that to a tangential double derivation $\{ a, -\}_\Pi \in \tD_A$ one can associate an element $| \{ a, -\}_\Pi | \in \tder(A)$ such that
$$
\rho\left( | \{ a, -\}_\Pi | \right) = |i(\{ a, -\}_\Pi)| = \{ |a|, -\}_\Pi \in \der(A).
$$
For simplicity, with the assumption that $\Pi$ is tangential,
we often use the same notation $\{|a|,-\}_{\Pi}$ for the lift $|\{ a, -\}_\Pi | \in \tder(A)$.


%


\begin{dfn}
\label{dfn:divp}
Let $\Pi$ be a tangential double bracket on $A$.
\begin{enumerate}
\item[(i)]
Define the map $\tdiv_{\Pi}:|A|\to |A| \otimes |A|$ by
$\tdiv_{\Pi}(|a|):=\tdiv(\{ |a|,- \}_{\Pi})$ for $|a|\in |A|$.
\item[(ii)]
Define the maps $\tdiv^T_{\Pi}: A\to |A|\otimes A$ and
$\underline{\tdiv}^T_{\Pi}:A\to A\otimes |A|$ by
$\tdiv^T_{\Pi}(a):=\tdiv^T(\Pi(a,-))$ and
$\underline{\tdiv}^T_{\Pi}(a):=\underline{\tdiv}^T(\Pi(a,-))$ for $a\in A$.
\end{enumerate}
\end{dfn}

Another description of the map $\tdiv^T_{\Pi}$ is as follows.
First, we consider the unique extension of $\Pi$ to a tangential double bracket
$\Pi^T:A^T\otimes A^T\to A^T\otimes A^T$ by imposing the condition
$\Pi^T(T,-)=\Pi^T(-,T)=0$.
Note that $\Pi^T(Ta,-)=T\,\Pi(a,-)$ for any $a\in A$.
To simplify the notation, we drop $T$ and denote the corresponding bracket by $\{-,-\}_{\Pi}$.
Then, it is easy to check that
\[
\tdiv^T_{\Pi}(a)=p_1(\tdiv(\{|Ta|,-\}_{\Pi})),
\]
where $\{|Ta|,-\}_{\Pi}\in \der(A^T)$ is lifted to an element of
$\tder(A^T)$, and $\tdiv:\tder(A^T)\to |A^T|\otimes |A^T|$ is the tangential divergence on $\tder(A^T)$.

Note that $\tdiv_{\Pi}$, $\tdiv^T_{\Pi}$, and $\underline{\tdiv}^T_{\Pi}$
are $\K$-linear in $\Pi$, i.e., $\tdiv_{\Pi_1+\Pi_2}=\tdiv_{\Pi_1}+\tdiv_{\Pi_2}$
for $\Pi_1$ and $\Pi_2$ tangential, etc.

\begin{prop}
\label{prop:prodP}
Let $\Pi$ be a tangential double bracket on $A$. Then
for any $a,b\in A$,
\begin{align}
\tdiv^T_{\Pi}(ab)=&
\tdiv^T_{\Pi}(a)(1\otimes b)+(1\otimes a)\tdiv^T_{\Pi}(b)+(|\ |\otimes 1)\Pi(a,b), \label{eq:tdivab1} \\
\underline{\tdiv}^T_{\Pi}(ab)=&
\underline{\tdiv}^T_{\Pi}(a)(b\otimes 1)+(a\otimes 1)\underline{\tdiv}^T_{\Pi}(b)
+(1\otimes |\ |)\Pi(b,a), \label{eq:tdivab2}
\end{align}
and for any $a\in A$,
\begin{equation}
\label{eq:tdiv(a)}
\tdiv_{\Pi}(|a|)=(1\otimes |\ |)\tdiv^T_{\Pi}(a)+
(|\ |\otimes 1)\underline{\tdiv}^T_{\Pi}(a).
\end{equation}
\end{prop}

\begin{proof}
By Proposition \ref{prop:aphi}, we compute
\begin{align*}
& \tdiv^T_{\Pi}(ab)= \tdiv^T( \Pi(ab,-)) \\
=& \tdiv^T (\Pi(a,-)b+a\Pi(b,-)) \\
=& \tdiv^T (\Pi(a,-))(1\otimes b)+(|\ |\otimes 1)\Pi(a,b)+
(1\otimes a)\tdiv^T(\Pi(b,-)) \\
=& \tdiv^T_{\Pi}(a)(1\otimes b)+(1\otimes a)\tdiv^T_{\Pi}(b)
+(|\ |\otimes 1)\Pi(a,b).
\end{align*}
This proves (\ref{eq:tdivab1}).
To prove (\ref{eq:tdivab2}), we compute
\begin{align*}
& \underline{\tdiv}^T_{\Pi}(ab)= \underline{\tdiv}^T( \Pi(ab,-)) \\
=& \underline{\tdiv}^T (\Pi(a,-)b+a\Pi(b,-)) \\
=& \underline{\tdiv}^T (\Pi(a,-))(b\otimes 1)
+(a\otimes 1)\underline{\tdiv}^T(\Pi(b,-))+(1\otimes |\ |)\Pi(b,a)
 \\
=& \underline{\tdiv}^T_{\Pi}(a)(b\otimes 1)
+(a\otimes 1)\underline{\tdiv}^T_{\Pi}(b)
+(|\ |\otimes 1)\Pi(b,a).
\end{align*}
Finally, (\ref{eq:tdiv(a)}) follows from Proposition \ref{prop:rel_div}.
\end{proof}

\begin{rem}
\label{rem:any}
The notion of a double bracket (Definition \ref{dfn:db}) and its auxiliary operations make sense for any associative $\K$-algebra.
In \S \ref{sec:GTLB}, we consider a topologically defined double bracket on the group ring $\K \pi$, where $\pi$ is the fundamental group of an oriented surface.
\end{rem}

\subsection{The KKS double bracket}

In this section, we consider an important example of a double bracket: the Kirillov-Kostant-Souriau (KKS) double bracket. We will use the map \eqref{eq:phi1phi2} to identify double brackets with elements of $(DA)_2$. In this identification, the symbol $\otimes$ stands for $\otimes_A$. 
The KKS double bracket is given by the following formula:

\begin{equation}
\Pi_{\rm KKS} =  \sum_{i=1}^n  | \partial_i \otimes {\rm ad}_{x_i} \partial_i |.
\end{equation}

\begin{lem} \label{lem:KKS_task}
The double bracket $\Pi_{\rm KKS}$ is tangential and skew-symmetric.
\end{lem}
\begin{proof}
The skew-symmetry \eqref{eq:sks} is obvious from the following computation:
$$
\Pi_{\rm KKS} =
\sum_i |\partial_i \otimes (x_i \partial_i - \partial_i x_i) |= - \sum_i
|(x_i \partial_i  - \partial_i x_i) \otimes \partial_i|,
$$
where we have used the cyclicity property for $x_i$ under the $| \cdot |$ sign.
The fact that $\Pi_{\rm KKS}$ is tangetial can be checked on generators:
$$
\Pi_{\rm KKS}(x_i, -)= {\rm ad}_{x_i} \partial_i \in \tD_A.
$$
\end{proof}

Denote $x_0=-\sum_{i=1}^n x_i$.

\begin{lem}  \label{prop:KKS}
For the KKS double bracket, we have
\begin{align*}
& \{ |h(x_i)|, a \}_{\KKS} =\{a,|h(x_i)|\}_{\KKS}= 0, \\
& \{ x_0, a \}_{\KKS} = -\phi_0(a), \\
& \{ x_0, |a|\}_{\KKS}=\{ |a|, x_0\}_{\rm KKS}=0, \\
& \{ |h(x_0)|, a \}_{\rm KKS}= -\{a,|h(x_0)|\}_{\KKS}=  [a, \dot{h}(x_0)],
\end{align*}
where $a \in A$, $i=1,\ldots,n$, $h(z) \in \mathbb{K}[[z]]$ and $\dot{h}(z)$ its derivative.
\end{lem}

\begin{proof}
First of all, note that $\Pi_{\KKS}$ is skew-symmetric by Lemma \ref{lem:KKS_task} so that we can use Lemma \ref{lem:if_skew}.

For the first equality, note that $\{ x_i, -\}_{\rm KKS} = [x_i, \partial_i]$. Then, for any formal power series $h(z) \in \mathbb{K}[[z]]$, we have
$$
\{ |h(x_i)|, -\}_{\rm KKS} = |\dot{h}(x_i)[x_i, \partial_i]| = 0,
$$
where we have used the cyclic property under the $| \cdot |$ sign.

For the second equality, we have
$$
\{ x_0, -\}_{\rm KKS} = -\{ \sum_i x_i, -\}_{\rm KKS} = - \sum_i [x_i, \partial_i] = - \phi_0.
$$
This also proves $\{|a|,x_0\}_{\KKS}=0$.

Finally, for the last equality we compute
$$
\{ |h(x_0)|, -\}_{\rm KKS} = - |\dot{h}(x_0) \phi_0|
$$
which is the inner derivation with generator $-\dot{h}(x_0)$ (see Example \ref{example:inner}), as required.
\end{proof}

\begin{lem} \label{lem:maKKS}
For $h(z)\in \K[[z]]$ and $a\in A$, let $h' \otimes h'' = \tde h(x_0)$. Then
\begin{align*}
\{|h'|,ah''\}_{\KKS} = -\dot{h}(\ad_{x_0})(a) + \dot{h}(0) a, \\
\{|h'|,h''a\}_{\KKS} = \dot{h}(-\ad_{x_0})(a) - \dot{h}(0) a.
\end{align*}
\end{lem}

\begin{proof}
One computes
\begin{align*}
    \{|h'|,ah''\}_{\KKS} &= \{|h'|, a \}_\KKS \, h'' + a \{|h'|, h'' \}_\KKS \\
    &= [a, (\dot{h})'] \, (\dot{h})'' \\
    & = a (\dot{h})'(\dot{h})'' - (\dot{h})'a (\dot{h})'' \\
    &= a \dot{h}(0) - \dot{h}(\ad_{x_0})a,
\end{align*}
where in the second line we have used Lemma \ref{prop:KKS} and in the last line we have used that $(\dot{h})' (\dot{h})''=\varepsilon(\dot{h}(x_0))=\dot{h}(0)$ and $(\dot{h})'a(\dot{h})''=\dot{h}(\ad_{x_0})a$.
The other equality can be proven in a similar way.
\end{proof}

One can obtain an explicit formula for the KKS double bracket on a pair of words $z=z_1\cdots z_l, w=w_1\cdots w_m$, where $z_j,w_k\in \{x_i\}_{i=1}^n$. Indeed,

$$
\begin{array}{lll}
\Pi_{\rm KKS}(z,w) & = & \sum_k w_1 \dots w_{k-1} \Pi_{\rm KKS}(z, w_k) w_{k+1} \dots w_m \\
& = & \sum_{j,k}
w_1\cdots w_{k-1} \left( (z_1\cdots z_{j-1})*
\Pi_{\rm KKS}(z_j, w_k)*(z_{j+1}\cdots z_l) \right)
w_{k+1}\cdots w_m.
\end{array}
$$
Since $\Pi_{\rm KKS}(z_j, w_k)=\delta_{z_j, w_k} (1 \otimes z_j - z_j \otimes 1)$, we obtain
$$
\begin{array}{rcl}
\Pi_{\rm KKS}(z,w) & = & \sum_{j,k} \delta_{z_j, w_k}
\left( w_1\cdots w_{k-1}z_{j+1}\cdots z_l \otimes z_1\cdots z_j w_{k+1}\cdots w_m \right. \\
&  & \hspace{6em} \left. -w_1\cdots w_{k-1}z_j\cdots z_l \otimes z_1\cdots z_{j-1} w_{k+1}\cdots w_m \right).
\end{array}
$$

\begin{rem}
The KKS double bracket is an algebraic counterpart of the version of the homotopy intersection form $\kappa$ introduced in \cite{KK15}, \cite{MT14}. To align the notation, we sometimes use the symbol $\kappa^{\rm alg}: A\otimes A \to A \otimes A$ for the KKS double bracket.
Double brackets can be naturally associated to quivers \cite{vdB08}. In this formalism, the KKS double bracket is associated to the star-shaped quiver with one internal vertex and $n$ external vertices.
\end{rem}

\begin{rem}
The double bracket $\Pi_{\rm KKS}$ verifies the van den Bergh version of the Jacobi identity \cite{vdB08}. In particular, this implies that 
$$
(|a|, |b|) \mapsto \{ |a|, |b|\}_{\rm KKS}
$$
is a Lie bracket on $|A|$.
In the framework of double brackets associated to quivers, this Lie algebra is known as the necklace Lie algebra \cite{BB02}, \cite{Gi01}. See also \cite{MW} for a related structure in higher genus.
\end{rem}

\begin{thm}\label{thm:center}
The center of the Lie algebra $(|A|, \{-,-\}_{\KKS})$ is spanned by 
$|\K[[x_i]]|$, $0\le i\le n$,
$$
Z(|A|, \{-,-\}_{\KKS}) = \sum^n_{i=0}|\K[[x_i]]|.
$$
\end{thm}

\begin{proof}
A proof using Poisson geometry of representation varieties can be obtained with the techniques of \cite{CBEG07}. In Appendix A, we give a direct algebraic proof of this statement.
\end{proof}

Next, we compute the maps $\tdiv_\Pi$ and $\tdiv_\Pi^T$ for the KKS double bracket. We will use the following shorthand notation: $\tdiv_{\rm KKS}$ and $\tdiv^T_{\rm KKS}$. A similar style notation will be later used for other double brackets.

\begin{prop} \label{prop:KKS_div}
For $z=z_1z_2 \cdots z_m$ with  $z_j\in \{ x_i\}_{i=1}^n$, we have
\begin{align*}
\tdiv_{\rm KKS}(|z|)=& -\sum_{j<k} \delta_{z_j,z_k}
\left(
|z_j \cdots z_{k-1}| \wedge |z_{k+1}\cdots z_m z_1 \cdots z_{j-1}| \right. \\
& \hspace{6em} \left. +|z_k \cdots z_m z_1 \cdots z_{j-1}| \wedge |z_{j+1}\cdots z_{k-1}| \right),
\end{align*}
\begin{align*}
\tdiv^T_{\rm KKS}(z)=& \sum_{j<k} \delta_{z_j,z_k}
\left(
|z_{j+1}\cdots z_{k-1}| \otimes z_1\cdots z_j z_{k+1} \cdots z_m \right. \\
& \hspace{6em} \left. -|z_j\cdots z_{k-1}|\otimes z_1\cdots z_{j-1}z_{k+1}\cdots z_m \right).
\end{align*}
\end{prop}

\begin{proof}
Since $\{z,-\}_{\KKS}=\sum_j z_1\cdots z_{j-1} ({\rm ad}_{z_j}\pa_{z_j}) z_{j+1}\cdots z_m$,
the $i$th component of $\{|z|,-\}_{\KKS} \in \tder(A)$ is
$-\sum_j \delta_{x_i,z_j} z_{j+1}\cdots z_m z_1\cdots z_{j-1}$.
Therefore,
\begin{align*}
\tdiv_{\KKS}(|z|)
=&
-\sum_j | \tde(z_j) \pa_{z_j} (z_{j+1}\cdots z_m z_1 \cdots z_{j-1})| \\
=& -\sum_{j<k} \delta_{z_j,z_k}
\left(
|z_j \cdots z_{k-1}| \otimes |z_{k+1} \cdots z_m z_1 \cdots z_{j-1}| \right. \\
& \hspace{6em} \left. -|z_{j+1} \cdots z_{k-1}| \otimes |z_{k+1} \cdots z_m z_1 \cdots z_j|
\right) \\
& \hspace{1em} -\sum_{k<j} \delta_{z_j,z_k}
\left(
|z_j\cdots z_mz_1 \cdots z_{k-1}| \otimes |z_{k+1} \cdots z_{j-1}| \right. \\
& \hspace{7em} \left. -|z_{j+1}\cdots z_m z_1\cdots z_{k-1}| \otimes |z_{k+1} \cdots z_j|
\right) \\
=& -\sum_{j<k} \delta_{z_j,z_k}
\left(
|z_j\cdots z_{k-1}| \wedge |z_{k+1} \cdots z_mz_1 \cdots z_{j-1}| \right. \\
& \hspace{6em} \left. +|z_k\cdots z_mz_1\cdots z_{j-1}| \wedge |z_{j+1}\cdots z_{k-1}|
\right).
\end{align*}

For $\tdiv^T_{\KKS}$, using (\ref{eq:tdiv^T}) we compute
\begin{align*}
\tdiv^T_{\KKS}(z) =&
(|\ |\otimes 1)\left(
\sum_j (1\otimes z_1 \cdots z_j-z_j\otimes z_1\cdots z_{j-1})
\pa_{z_j}(z_{j+1}\cdots z_m)
\right) \\
=& \sum_{j<k} \delta_{z_j,z_k}
\left(
|z_{j+1}\cdots z_{k-1}| \otimes z_1\cdots z_j z_{k+1} \cdots z_m \right. \\
& \hspace{6em} \left. -|z_j\cdots z_{k-1}|\otimes z_1\cdots z_{j-1}z_{k+1}\cdots z_m \right).
\end{align*}
\end{proof}
 
\begin{rem}
The operations $\tdiv_{\rm KKS}$ and $\tdiv^T_{\rm KKS}$ are algebraic counterparts of the Turaev cobracket $\delta$ and of the operation $\mu$ (a self-intersection map introduced in \cite{KK15}). To align the notation, we sometimes use the following symbols: $\delta^{\rm alg}= - \tdiv_{\rm KKS}, \mu^{\rm alg} = \tdiv^T_{\rm KKS}$.  
\end{rem}

\begin{rem}
The vector space $|A|$ endowed with the Lie bracket induced by $\Pi_{\rm KKS}$ and with the operation $\delta^{\rm alg}: |A| \to \wedge^2 |A|$ is in fact a Lie bialgebra. This Lie bialgebra was introduced and studied in \cite{Schedler} (see also \cite{MW} for a similar structure in higher genus).
It can be viewed as an algebraic counterpart of the Goldman-Turaev Lie bialgebra for a surface of genus zero defined in the next section.
\end{rem}

\subsection{Double brackets $\Pi_s$}

Another family of examples are the double brackets of the form
$$
\Pi_s=| \phi_0 \otimes s(-{\rm ad}_{x_0}) \phi_0| = |s({\rm ad}_{x_0}) \phi_0 \otimes \phi_0|,
$$
for $s=s(z) \in \mathbb{K}[[z]]$ a formal power series. In particular, for $s=1$ we obtain $\Pi_1 = |\phi_0 \otimes \phi_0|$. For the later use, we will need the following Sweedler style notation:
$$
\tde(s(-x_0))=(1\otimes \iota)\Delta(s(-x_0))= s'\otimes s''
 \in A \otimes A.
$$

\begin{prop}
\label{prop:add'}
For any $s$, the double bracket $\Pi_s$ is tangential. For any $a,b\in A$, we have
\[
\Pi_s(a,b)=
s''a\otimes s'b-bs''a\otimes s'
-s''\otimes as'b+bs''\otimes as'.
\]
The tangential derivation $\{|a|,-\}_s\in \tder(A)$ vanishes for all $a \in A$.
All the bracket operations $\{-,-\}_s: |A| \otimes A \to A$,
$|A|\otimes |A| \to |A|$ and $A\otimes |A|\to A$
are identically zero.
\end{prop}

\begin{proof}
Let $a\in A$.
By Lemma \ref{lem:phi_0} and (\ref{eq:ph1ph2}), we have
\begin{equation}
\label{eq:add'a}
\Pi_s(a,-)
=(s(-\ad_{x_0})\phi_0) a-a (s(-\ad_{x_0})\phi_0)
=s'\phi_0 s''a-as'\phi_0 s''.
\end{equation}
On the one hand, this shows that $\Pi_s$ is tangential.
On the other hand, for  $b\in A$ we have
$$
\Pi_s(a,b)=s'*(1\otimes b-b\otimes 1)*(s''a)
-(as')*(1\otimes b-b\otimes 1)*s''.
$$
Expanding the right hand side of this equation, we obtain the second statement in the proposition.

The equation (\ref{eq:add'a}) shows that the tangential derivation $\{|a|,-\}_s \in \tder(A)$ is an inner derivation with generator
$s''as' - s''as'=0$.
The last statement on the bracket follows from the second statement. 
\end{proof}

\begin{prop}
\label{prop:add'_div}
For any $a\in A$,
$\tdiv_s(|a|)=0$ and 
$$
\tdiv^T_s(a)=
|s''|\otimes as'-|s''a|\otimes s'.
$$
\end{prop}

\begin{proof}
By Proposition \ref{prop:add'}, the tangential derivation $\{|a|,-\}_s\in \tder(A)$ vanishes. Hence, $\tdiv_s(|a|)=0$ for all $a \in A$.
By Lemma \ref{lem:phi_0} and Proposition \ref{prop:aphi},
\begin{align*}
\tdiv^T_s(a) =& \tdiv^T(s'\phi_0 s''a-as'\phi_0 s'') \\
=& \tdiv^T(s'\phi_0)(1\otimes s''a)+(|\ |\otimes 1)(s'\phi_0)(s''a) \\
& -\tdiv^T(as'\phi_0)(1\otimes s'')-(|\ |\otimes 1)(as'\phi_0)(s'') \\
=& (|\ |\otimes 1)\left(
s'* (1\otimes s''a-s''a\otimes 1)
-(as')* (1\otimes s''-s''\otimes 1)
\right) \\
=& (|\ |\otimes 1)(1\otimes s's''a-s''a\otimes s'
-1\otimes as's''+s''\otimes as') \\
=& |s''|\otimes as'-|s''a|\otimes s'.
\end{align*}
Here we have used the facts that  $\tdiv^T(\phi_0)=0$
(this can be seen directly from (\ref{eq:tdiv^T})), and
that $s's''=\varepsilon(s(-x_0)) \in \mathbb{K}$.
\end{proof}

\section{The Goldman-Turaev Lie bialgebra}
\label{sec:GTLB}

In this section, we recall the notions of Goldman bracket, Turaev cobracket and their upgrades to the group algebra of the fundamental group. All these operations are defined in terms  of intersections and self-intersections of curves on a surface $\Sigma$.

Let $\Sigma$ be a connected oriented smooth 2-dimensional manifold with non-empty boundary $\partial \Sigma$.
Fix a point $*\in \pa \Sigma$, and
let $\pi:=\pi_1(\Sigma,*)$ be the fundamental group of $\Sigma$ with base point $*$. 
The group ring $\K \pi$ has a Hopf algebra structure whose coproduct $\Delta$, augmentation $\varepsilon$, and antipode $\iota$ are defined by the following formulas: for $\alpha\in \pi$,
\[
\Delta(\alpha):=\alpha\otimes \alpha, \quad \varepsilon(\alpha):=1, \quad \iota(\alpha):=\alpha^{-1}.
\]

By a \emph{curve}, we mean a path or a loop on $\Sigma$.
In what follows, we consider several operations on (homotopy classes of) curves on $\Sigma$.
Let $I$ be the unit interval $[0,1]$ or the circle $S^1$, and let $\alpha$ be a smoothly immersed curve, i.e., a $C^{\infty}$-immersion $\alpha: I \to \Sigma$.
For distinct points $s,t \in I$, where we further assume that $s<t$ if $I=[0,1]$, we define the curve $\alpha_{st}$ to be the restriction of $\alpha$ to the interval $[s,t]$.
For $t\in I$, let $\dot{\alpha}(t)\in T_{\alpha(t)}\Sigma$ denote the velocity vector of $\alpha$ at $t$.
If $p\in \alpha(I)$ is a simple point of $\alpha$ and $t=\alpha^{-1}(p)$, we also write $\dot{\alpha}_p$ for $\dot{\alpha}(t)$.
For paths $\alpha,\beta:[0,1]\to \Sigma$, their concatenation $\alpha\beta$ is defined if and only if $\alpha(1)=\beta(0)$, and is the path first traversing $\alpha$, then $\beta$.
If there is no fear of confusion, we use the same letter for a curve and its (regular) homotopy class.

For a path $\gamma:[0,1]\to \Sigma$,
its inverse $\overline{\gamma}:[0,1]\to \Sigma$ is defined by
$\overline{\gamma}(t):=\gamma(1-t)$.

Let $p\in \Sigma$ and let $\vec{a},\vec{b}\in T_p\Sigma$ be linearly independent tangent vectors.
The \emph{local intersection number} $\varepsilon(\vec{a},\vec{b})\in \{ \pm 1\}$ of $\vec{a}$ and $\vec{b}$ is defined as follows.
If the ordered pair $(\vec{a},\vec{b})$ is a positively oriented frame of $\Sigma$, then $\varepsilon(\vec{a},\vec{b}):=+1$; otherwise,  $\varepsilon(\vec{a},\vec{b}):=-1$.

Take an orientation preserving embedding
$\nu: [0,1]\to \pa \Sigma$
with $\nu(1)=*$, and set $\bullet:=\nu(0)$.
Using $\nu$, we obtain the following isomorphism of groups:
\begin{equation}
\label{eq:bul}
\pi \xrightarrow{\cong} \pi_1(\Sigma,\bullet), \quad \alpha\mapsto \nu \alpha \overline{\nu}.
\end{equation}

\subsection{The Goldman bracket and the double bracket $\kappa$}

We recall  an operation which measures the intersection of two based loops on $\Sigma$ (see \cite{KK15} \S 4.2 and \cite{MT14} \S 7).
Let $\alpha$ and $\beta$ be loops on $\Sigma$ based at $\bullet$ and $*$, respectively.
We assume that $\alpha$ and $\beta$ are $C^{\infty}$-immersions with their intersections consisting of transverse double points.
Then set
\begin{equation}
\label{eq:kap}
\kappa(\alpha,\beta):=
-\sum_{p\in \alpha \cap \beta}
\varepsilon(\dot{\alpha}_p,\dot{\beta}_p) \,
\beta_{*p}\alpha_{p\bullet}\nu \otimes
\overline{\nu}\alpha_{\bullet p}\beta_{p *} \in \K \pi \otimes \K \pi.
\end{equation}
The result only depends on the homotopy classes of $\alpha$ and $\beta$, and using the isomorphism (\ref{eq:bul}), we obtain a $\K$-linear map $\kappa: \K\pi \otimes \K\pi \to \K\pi \otimes \K\pi$.

\begin{prop}
The map
$\kappa: \K\pi \otimes \K\pi \to \K\pi \otimes \K\pi$
is a double bracket on $\K \pi$ in the sense of Remark \ref{rem:any}.
For any $x,y\in \K\pi$, we have
\begin{align}
& \kappa(y,x)=-\kappa(x,y)^{\circ}+xy\otimes 1+1\otimes yx
-x\otimes y-y\otimes x, \label{eq:k(y,x)} \\
&
\kappa(\iota(x),\iota(y))=(\iota\otimes \iota)\kappa(x,y)^{\circ}.
\label{eq:k(ix,iy)}
\end{align}
Here, $(u\otimes v)^{\circ}=v\otimes u$ for $u,v\in \K\pi$.
\end{prop}

\begin{proof}
For the proof of the first statement, see \cite{KK15} Lemma 4.3.1.
(Note that we use a different convention here. See also Remark
\ref{rem:KK15} below.)
The formula (\ref{eq:k(y,x)}) can be found in the proof of Lemma 7.2 in \cite{MT14}.
The formula (\ref{eq:k(ix,iy)}) can be seen directly from the defining formula (\ref{eq:kap}).
\end{proof}

\begin{rem}
Essentially the same operations as the map $\kappa$ were independently introduced by Papakyriakopoulos \cite{Papa} and Turaev \cite{Tu78}. Yet another version is the \emph{homotopy intersection form} $\eta: \K \pi \otimes \K \pi \to \K \pi$ introduced by Massuyeau and Turaev \cite{MT13}.
\end{rem}

Let $\hat{\pi}:=[S^1,\Sigma]$ be the set of
homotopy classes of free loops on $\Sigma$.
For any $p\in \Sigma$ (in particular, for $p=*$), let $|\ |: \pi_1(\Sigma,p)\to \hat{\pi}$ be the map obtained by forgetting the base point of a based loop.
Notice that the map $|\ |: \pi \to \hat{\pi}$ induces a $\K$-linear isomorphism
\[
\K\hat{\pi}\cong |\K\pi|,
\]
where $|\K\pi|=\K\pi/[\K\pi,\K\pi]$.
In this paper, we mainly use the notation $|\K \pi|$.

By equation \eqref{eq:bracket_on_|A|}, the double bracket $\kappa$ induces a bracket $\{-,-\}=\{-,-\}_\kappa$ on $\K \hat{\pi} \cong |\K \pi|$. 
To be more explicit, suppose that $\alpha$ and $\beta$ are smoothly immersed loops on $\Sigma$, such that their intersections consist of transverse double points.
For each intersection $p\in \alpha \cap \beta$,
let $\alpha_p$ and $\beta_p$ be the loops $\alpha$ and $\beta$ \emph{based at $p$}, and let $\alpha_p\beta_p\in \pi_1(\Sigma,p)$ be their concatenation.
Then from the formula (\ref{eq:kap}), we obtain
\[
\{\alpha,\beta\}=-\sum_{p\in \alpha\cap \beta}
\varepsilon(\dot{\alpha}_p,\dot{\beta}_p)
|\alpha_p \beta_p|.
\]
Note that $\{-,-\}$ is  \emph{minus}  the \emph{Goldman bracket} defined in  \cite{Go86}.

\subsection{Turaev cobracket and the self-intersection map $\mu$}

Choose inward non-zero tangent vectors $\xi_{\bullet}\in T_{\bullet}\Sigma$ and $\xi_*\in T_*\Sigma$, respectively.
Let $\pi^+=\pi_{\bullet *}^+$ be the set of regular homotopy classes relative to the boundary $\{0,1\}$ of $C^{\infty}$-immersions
$\gamma: ([0,1],0,1) \to (\Sigma,\bullet,*)$ such that 
$\dot{\gamma}(0)=\xi_{\bullet}$ and $\dot{\gamma}(1)=-\xi_*$.
We define a group structure on $\pi^+$ as follows: for $\alpha,\beta\in \pi^+$, their product $\alpha \beta$ is
the insertion of a \emph{positive} monogon to
(a suitable smoothing of) $\alpha \overline{\nu} \beta$ (see Figure \ref{fig:prod}).
The map $\pi^+\to \pi, \gamma \mapsto \overline{\nu}\gamma$ is a surjective group homomorphism, and its kernel is an infinite cyclic group.

\begin{figure}
\begin{center}
\input{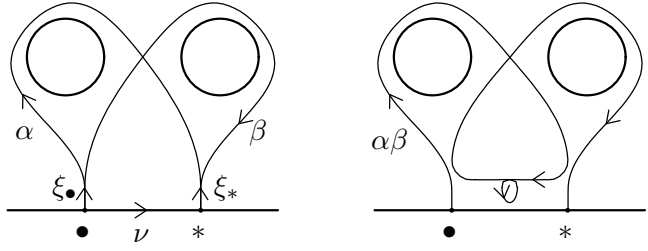}
\end{center}
\caption{the group structure on $\pi^+=\pi^+_{\bullet *}$}
\label{fig:prod}
\end{figure}

Let $\gamma:([0,1],0,1)\to (\Sigma,\bullet,*)$ be a $C^{\infty}$-immersion such that $\dot{\gamma}(0)=\xi_{\bullet}$ and $\dot{\gamma}(1)=-\xi_*$.
We assume that the self-intersections of $\gamma$ consist of transverse double points.
For each self-intersection $p$ of $\gamma$, let $\gamma^{-1}(p)=\{ t_1^p,t_2^p\}$ with $t_1^p<t_2^p$.
Then we define $\mu(\gamma)=\mu_{\bullet *}(\gamma)
\in |\K \pi|\otimes \K \pi$ by
\begin{equation}
\label{eq:mu(g)}
\mu(\gamma):=
-\sum_p \varepsilon(\dot{\gamma}(t_1^p),\dot{\gamma}(t_2^p)) \, 
|\gamma_{t_1^pt_2^p}|\otimes \overline{\nu}\gamma_{0t_1^p}\gamma_{t_2^p1},
\end{equation}
where the sum runs over all the self-intersections of $\gamma$.
Using the same argument as in \cite{KK15} Proposition 3.2.3,
we can show that the result depends only on the regular homotopy class of $\gamma$.
Since in this statement we deal with  \emph{regular homotopy} classes, we only have
to check the invariance under the second and third Reidemeister moves $(\omega 2)$ and $(\omega 3)$ (see \cite{Go86} \S 5).
In this way, we obtain a $\K$-linear map
$\mu=\mu_{\bullet *}:\K \pi^+\to |\K \pi| \otimes \K \pi$.

\begin{prop}
\label{prop:prod}
For any $\alpha,\beta\in \pi^+$,
\[
\mu(\alpha\beta)=
\mu(\alpha)(1\otimes \overline{\nu}\beta)+(1\otimes \overline{\nu}\alpha)\mu(\beta)+(|\ |\otimes 1)\kappa(\overline{\nu}\alpha,\overline{\nu}\beta).
\]
\end{prop}

\begin{proof}
The proof is essentially the same as the proof of Lemma 4.3.3
in \cite{KK15}, so we omit the detail.
\end{proof}

\begin{rem}
In \cite{Tu78}, Turaev introduced essentially the same operation as $\mu$.
\end{rem}

We denote by $\hat{\pi}^+$ the set of regular homotopy classes of immersed free loops on $\Sigma$.
Let $\alpha:S^1\to \Sigma$ be a $C^{\infty}$-immersion whose self-intersections consist of transverse double points.
Setting $D_{\alpha}:=\{ (t_1,t_2)\in S^1\times S^1 \mid t_1\neq t_2, \alpha(t_1)=\alpha(t_2) \}$, we define
\begin{equation}
\label{eq:d^+}
\delta^+(\alpha):=
\sum_{(t_1,t_2)\in D_{\alpha}}
\varepsilon(\dot{\alpha}(t_1),\dot{\alpha}(t_2)) \,
|\alpha_{t_1t_2}| \otimes |\alpha_{t_2t_1}| \in |\K \pi|\otimes |\K \pi|.
\end{equation}
The right hand side only depends on the regular homotopy class of $\gamma$,
and we obtain a $\K$-linear map $\delta^+: \mathbb{K}\hat{\pi}^+\to |\K \pi|\otimes |\K \pi|$.

\begin{rem}
To show that $\delta^+(\alpha)$ is well-defined, 
we can use a similar argument to \cite{Tu91} \S 8. A regular homotopy version of the cobracket $\delta$ was first introduced by Turaev in \cite{Tu91} \S 18. Another version which is  closer to our definition was introduced in \cite{Ka15}.
\end{rem}

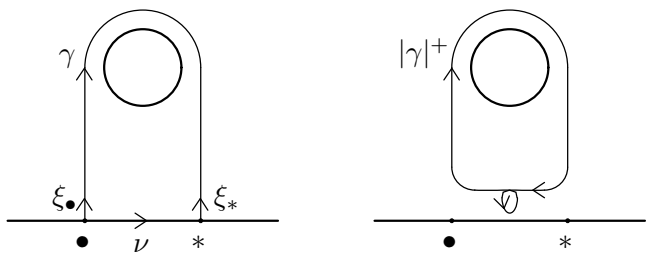
\begin{figure}
\begin{center}
{\unitlength 0.1in%
\begin{picture}( 33.0000, 11.4000)(  3.0000,-18.4000)%
%
\special{pn 13}%
\special{pa 3600 1800}%
\special{pa 2200 1800}%
\special{fp}%
%
\special{pn 13}%
\special{ar 2900 1000 200 200  0.0000000  6.2831853}%
%
\special{pn 4}%
\special{sh 1}%
\special{ar 3200 1800 10 10 0  6.28318530717959E+0000}%
%
\special{pn 4}%
\special{sh 1}%
\special{ar 2600 1800 10 10 0  6.28318530717959E+0000}%
\put(25.5000,-19.5000){\makebox(0,0)[lb]{$\bullet$}}%
\put(31.5000,-19.5000){\makebox(0,0)[lb]{$*$}}%
%
\special{pn 13}%
\special{pa 1700 1800}%
\special{pa 300 1800}%
\special{fp}%
%
\special{pn 13}%
\special{ar 1000 1000 200 200  0.0000000  6.2831853}%
%
\special{pn 4}%
\special{sh 1}%
\special{ar 1300 1800 10 10 0  6.28318530717959E+0000}%
%
\special{pn 4}%
\special{sh 1}%
\special{ar 700 1800 10 10 0  6.28318530717959E+0000}%
%
\special{pn 8}%
\special{pa 1020 1800}%
\special{pa 960 1760}%
\special{fp}%
\special{pa 1020 1800}%
\special{pa 960 1840}%
\special{fp}%
\put(6.5000,-19.5000){\makebox(0,0)[lb]{$\bullet$}}%
\put(12.5000,-19.5000){\makebox(0,0)[lb]{$*$}}%
\put(9.5000,-19.6500){\makebox(0,0)[lb]{$\nu$}}%
%
\special{pn 8}%
\special{pa 698 1684}%
\special{pa 657 1745}%
\special{fp}%
\special{pa 698 1684}%
\special{pa 737 1745}%
\special{fp}%
%
\special{pn 8}%
\special{pa 1298 1684}%
\special{pa 1257 1745}%
\special{fp}%
\special{pa 1298 1684}%
\special{pa 1337 1745}%
\special{fp}%
\put(5.3000,-17.4800){\makebox(0,0)[lb]{$\xi_{\bullet}$}}%
\put(13.6800,-17.4800){\makebox(0,0)[lb]{$\xi_*$}}%
%
\special{pn 8}%
\special{ar 2900 1680 40 80  6.2831853  3.1415927}%
%
\special{pn 8}%
\special{pa 2860 1680}%
\special{pa 2866 1651}%
\special{pa 2899 1640}%
\special{pa 2932 1640}%
\special{pa 2963 1641}%
\special{pa 3059 1641}%
\special{pa 3080 1640}%
\special{fp}%
%
\special{pn 8}%
\special{pa 2938 1680}%
\special{pa 2932 1651}%
\special{pa 2899 1640}%
\special{pa 2866 1640}%
\special{pa 2835 1641}%
\special{pa 2739 1641}%
\special{pa 2718 1640}%
\special{fp}%
%
\special{pn 8}%
\special{pa 3020 1640}%
\special{pa 3080 1600}%
\special{fp}%
\special{pa 3020 1640}%
\special{pa 3080 1680}%
\special{fp}%
%
\special{pn 8}%
\special{pa 2870 1736}%
\special{pa 2819 1685}%
\special{fp}%
\special{pa 2870 1736}%
\special{pa 2898 1670}%
\special{fp}%
%
\special{pn 8}%
\special{ar 1000 1000 300 300  3.1415927  6.2831853}%
%
\special{pn 8}%
\special{pa 700 1000}%
\special{pa 700 1800}%
\special{fp}%
%
\special{pn 8}%
\special{pa 1300 1000}%
\special{pa 1300 1800}%
\special{fp}%
%
\special{pn 8}%
\special{pa 698 1004}%
\special{pa 657 1065}%
\special{fp}%
\special{pa 698 1004}%
\special{pa 737 1065}%
\special{fp}%
%
\special{pn 8}%
\special{ar 2900 1000 300 300  3.1415927  6.2831853}%
%
\special{pn 8}%
\special{ar 2720 1520 120 120  1.5707963  3.1415927}%
%
\special{pn 8}%
\special{pa 2600 1520}%
\special{pa 2600 1000}%
\special{fp}%
%
\special{pn 8}%
\special{ar 3080 1520 120 120  6.2831853  1.5707963}%
%
\special{pn 8}%
\special{pa 3200 1520}%
\special{pa 3200 1000}%
\special{fp}%
%
\special{pn 8}%
\special{pa 2600 1004}%
\special{pa 2559 1065}%
\special{fp}%
\special{pa 2600 1004}%
\special{pa 2639 1065}%
\special{fp}%
\put(5.6500,-10.0000){\makebox(0,0)[lb]{$\gamma$}}%
\put(23.2000,-10.0400){\makebox(0,0)[lb]{$|\gamma|^+$}}%
\end{picture}}%
\end{center}
\caption{the closing operation $|\ |^+$}
\label{fig:clo}
\end{figure}

Consider the closing operation $|\ |^+ : \pi^+ \to \hat{\pi}^+$,
which maps $\gamma$ to an immersed loop
obtained by inserting a positive monogon to $\overline{\nu}\gamma$
(see Figure \ref{fig:clo}).
Let ${\bf 1}\in \hat{\pi}$ denote the class of a constant loop.
Then for any $\gamma \in \pi^+$,
\begin{equation}
\label{eq:d^+alt}
\delta^+( |\gamma|^+)
=-\alt (1\otimes |\ |)\mu(\gamma)+|\gamma| \wedge {\bf 1}.
\end{equation}
Here, $\alt: |\K \pi| \otimes |\K \pi| \to |\K \pi| \otimes |\K \pi|$ is the map sending $\alpha \otimes \beta$ to $\alpha \wedge \beta=\alpha \otimes \beta-\beta \otimes \alpha$.

The original \emph{Turaev cobracket} \cite{Tu91} is a map
$\delta: |\K \pi|/\K {\bf 1} \to (|\K \pi|/\K {\bf 1}) \otimes (|\K \pi|/\K {\bf 1})$.
It is defined using the formula (\ref{eq:d^+}), where the loop $\alpha$ represents an element of $\hat{\pi}$ (not of $\hat{\pi}^+$) and the right hand side is regarded as an element of  $(|\K \pi|/\K {\bf 1}) \otimes (|\K \pi|/\K {\bf 1})$.

\begin{rem}
\label{rem:GLA}
The Goldman bracket $[-,-]_{\rm G}=-\{-,-\}$ is a Lie bracket on $\K\hat{\pi}$, and $\K {\bf 1}$ is a Lie ideal \cite{Go86}.
The map $\delta$ is a Lie cobracket on the quotient Lie algebra $\K\hat{\pi}/\K {\bf 1}$, and the triple $(\K\hat{\pi}/\K {\bf 1},[-,-]_{\rm G},\delta)$ is a \emph{Lie bialgebra} \cite{Tu91}.
As was shown by Chas \cite{Cha04}, it is \emph{involutive} in the sense that $[-,-]_{\rm G}\circ \delta=0$.
\end{rem}

\begin{rem}
\label{rem:KK15}
The operations $\mu$ and $\kappa$ in this paper are different from the original versions in \cite{KK15}.
One can convert one version into the other by switching the first and second components of the target.
\end{rem}

\subsection{The case of genus zero}
\label{subsec:g0}

In this subsection, let $\Sigma$ be a surface of genus 0 with $n+1$ boundary components.
We label the boundary components of $\Sigma$ as
$\pa_0 \Sigma,\pa_1 \Sigma,\ldots,\pa_n \Sigma$,
so that $*\in \pa_0 \Sigma$.
Take a free generating system $\gamma_1,\ldots,\gamma_n$ for $\pi$, such that
each $\gamma_i$ is freely homotopic to the positively oriented boundary component $\pa_i \Sigma$
and the product $\gamma_1 \cdots \gamma_n$ is the negatively oriented boundary component $\pa_0 \Sigma$ (see Figure \ref{fig:gen}).

\begin{figure}
\begin{center}
{\unitlength 0.1in%
\begin{picture}( 24.6000, 14.4000)( 11.0000,-18.4000)%
%
\special{pn 13}%
\special{pa 3000 1800}%
\special{pa 1600 1800}%
\special{fp}%
%
\special{pn 13}%
\special{ar 2300 900 200 200  0.0000000  6.2831853}%
%
\special{pn 4}%
\special{sh 1}%
\special{ar 2600 1800 10 10 0  6.28318530717959E+0000}%
%
\special{pn 4}%
\special{sh 1}%
\special{ar 2000 1800 10 10 0  6.28318530717959E+0000}%
%
\special{pn 8}%
\special{pa 2320 1800}%
\special{pa 2260 1760}%
\special{fp}%
\special{pa 2320 1800}%
\special{pa 2260 1840}%
\special{fp}%
\put(19.5000,-19.5000){\makebox(0,0)[lb]{$\bullet$}}%
\put(25.5000,-19.5000){\makebox(0,0)[lb]{$*$}}%
\put(22.5000,-19.6500){\makebox(0,0)[lb]{$\nu$}}%
%
\special{pn 8}%
\special{pa 2000 1680}%
\special{pa 1959 1741}%
\special{fp}%
\special{pa 2000 1680}%
\special{pa 2039 1741}%
\special{fp}%
%
\special{pn 8}%
\special{pa 2598 1684}%
\special{pa 2557 1745}%
\special{fp}%
\special{pa 2598 1684}%
\special{pa 2637 1745}%
\special{fp}%
\put(18.3000,-17.4800){\makebox(0,0)[lb]{$\xi_{\bullet}$}}%
\put(26.6800,-17.4800){\makebox(0,0)[lb]{$\xi_*$}}%
%
\special{pn 13}%
\special{ar 1600 900 200 200  0.0000000  6.2831853}%
%
\special{pn 13}%
\special{ar 3000 900 200 200  0.0000000  6.2831853}%
%
\special{pn 13}%
\special{ar 1600 1400 500 400  1.5707963  3.1415927}%
%
\special{pn 13}%
\special{ar 3000 1400 500 400  6.2831853  1.5707963}%
%
\special{pn 13}%
\special{ar 3000 800 500 400  4.7123890  6.2831853}%
%
\special{pn 13}%
\special{ar 1600 800 500 400  3.1415927  4.7123890}%
%
\special{pn 13}%
\special{pa 1600 400}%
\special{pa 3000 400}%
\special{fp}%
%
\special{pn 13}%
\special{pa 3500 800}%
\special{pa 3500 1400}%
\special{fp}%
%
\special{pn 13}%
\special{pa 1100 1400}%
\special{pa 1100 800}%
\special{fp}%
%
\special{pn 13}%
\special{pa 2000 1800}%
\special{pa 2000 1680}%
\special{fp}%
%
\special{pn 8}%
\special{ar 1430 1500 100 100  6.2831853  4.7123890}%
%
\special{pn 8}%
\special{pa 1530 1500}%
\special{pa 1470 1530}%
\special{fp}%
\special{pa 1530 1500}%
\special{pa 1565 1555}%
\special{fp}%
%
\special{pn 13}%
\special{pa 2600 1800}%
\special{pa 2600 1680}%
\special{fp}%
%
\special{pn 8}%
\special{ar 1600 900 280 280  0.0000000  6.2831853}%
%
\special{pn 8}%
\special{ar 2300 900 280 280  0.0000000  6.2831853}%
%
\special{pn 8}%
\special{ar 3000 900 280 280  0.0000000  6.2831853}%
%
\special{pn 8}%
\special{pa 3000 1180}%
\special{pa 2600 1680}%
\special{fp}%
%
\special{pn 8}%
\special{pa 2600 1680}%
\special{pa 2300 1180}%
\special{fp}%
%
\special{pn 8}%
\special{pa 2600 1680}%
\special{pa 1600 1180}%
\special{fp}%
%
\special{pn 8}%
\special{pa 2046 1400}%
\special{pa 2076 1467}%
\special{fp}%
\special{pa 2046 1400}%
\special{pa 2118 1399}%
\special{fp}%
%
\special{pn 8}%
\special{pa 2432 1400}%
\special{pa 2430 1474}%
\special{fp}%
\special{pa 2432 1400}%
\special{pa 2497 1430}%
\special{fp}%
%
\special{pn 8}%
\special{pa 2820 1404}%
\special{pa 2750 1428}%
\special{fp}%
\special{pa 2820 1404}%
\special{pa 2815 1475}%
\special{fp}%
%
\special{pn 8}%
\special{pa 1320 880}%
\special{pa 1279 941}%
\special{fp}%
\special{pa 1320 880}%
\special{pa 1359 941}%
\special{fp}%
%
\special{pn 8}%
\special{pa 2020 880}%
\special{pa 1979 941}%
\special{fp}%
\special{pa 2020 880}%
\special{pa 2059 941}%
\special{fp}%
%
\special{pn 8}%
\special{pa 2720 880}%
\special{pa 2679 941}%
\special{fp}%
\special{pa 2720 880}%
\special{pa 2759 941}%
\special{fp}%
\put(15.6000,-5.6000){\makebox(0,0)[lb]{$\gamma_1$}}%
\put(22.4000,-5.6000){\makebox(0,0)[lb]{$\gamma_2$}}%
\put(29.6000,-5.6000){\makebox(0,0)[lb]{$\gamma_3$}}%
\put(35.6000,-13.6000){\makebox(0,0)[lb]{$\partial_0 \Sigma$}}%
\end{picture}}%
\end{center}
\caption{a surface of genus $0$, embedded in $\mathbb{R}^2$ ($n=3$)}
\label{fig:gen}
\end{figure}
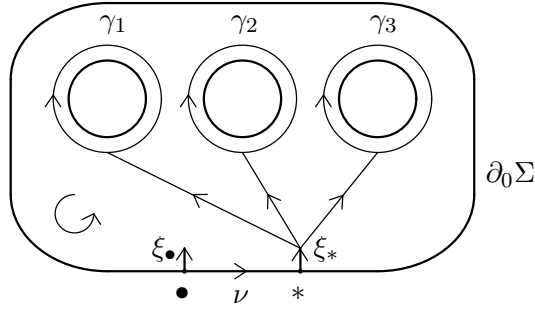

Figure \ref{fig:kgg} shows the following:

\begin{prop}
\label{prop:k(ga,ga)}
\begin{align*}
&\kappa(\gamma_i,\gamma_i)=
1\otimes \gamma_i^2-\gamma_i\otimes \gamma_i \quad \text{for any $i$}, \\
&\kappa(\gamma_i,\gamma_j)=0 \quad \text{if $i<j$},  \\
&\kappa(\gamma_i,\gamma_j)=
1\otimes \gamma_i\gamma_j+\gamma_j\gamma_i\otimes 1
-\gamma_i\otimes \gamma_j-\gamma_j\otimes \gamma_i
\quad \text{if $i>j$.}
\end{align*}
\end{prop}

\begin{figure}
\begin{center}
\input{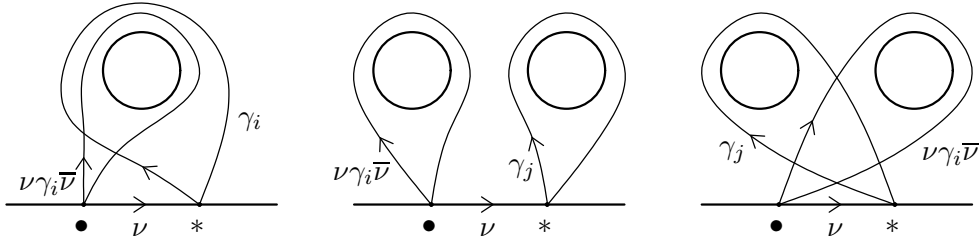}
\end{center}
\caption{the proof of Proposition \ref{prop:k(ga,ga)}}
\label{fig:kgg}
\end{figure}

Choose an orientation preserving embedding $\Sigma \to \mathbb{R}^2$ such that $\pa_0 \Sigma$ is the outermost boundary component of  $\Sigma$, as in Figure \ref{fig:gen}.
Note that such an embedding is unique up to isotopy.
The standard framing of $\mathbb{R}^2$ induces a framing
of $\Sigma$.
We may assume that the velocity vector of $\nu$ does not rotate
with respect to this framing.
Then we can define the rotation number function
\[
\rot=\rot_{\bullet *}: \pi^+ \to \frac{1}{2}+\Z.
\]
It satisfies the product formula 
$\rot(\alpha \beta)=\rot(\alpha)+\rot(\beta)+(1/2)$ for any $\alpha,\beta\in \pi^+$.

The surjective homomorphism $\pi^+ \to \pi, \gamma \mapsto \overline{\nu}\gamma$ has a section $i=i_{\bullet *}:\pi\to \pi^+$ defined by the condition that $\rot(i(\gamma))=-1/2$ for any $\gamma\in \pi$.
Composing $\mu: \K \pi^+\to |\K \pi| \otimes \K\pi$ with the section $i$, we obtain a $\K$-linear map
$\mu=\mu_{\bullet *}: \K\pi \to |\K \pi| \otimes \K\pi$
(again denoted by $\mu$).

\begin{prop}
\label{prop:prod2}
The map $\mu: \K\pi \to |\K \pi| \otimes \K\pi$ satisfies
the product formula
\[
\mu(xy)=
\mu(x)(1\otimes y)+(1\otimes x)\mu(y)+(|\ |\otimes 1)\kappa(x,y)
\]
for any $x,y\in \K \pi$, and we have $\mu(\gamma_i)=0$ for any $i$.
Moreover, these two properites characterise the map $\mu$.
\end{prop}

\begin{proof}
The product formula follows from Proposition \ref{prop:prod}.
Notice that a suitable smoothing of
$\nu \gamma_i$ has no self-intersections and that its rotation number
is $-1/2$.
Thus $i(\gamma_i)=\nu \gamma_i$ and $\mu(\gamma_i)=0$.
The last statement follows from the fact that $\pi$ is generated by $\{\gamma_i\}_i$.
\end{proof}

\begin{prop} \label{prop:cob}
The  composition map
\[
\pi \xrightarrow{i} \pi^+ \xrightarrow{|\ |^+} \hat{\pi}^+
\xrightarrow{\delta^+} |\K \pi|\otimes |\K \pi|.
\]
descends to a map $\delta^+: \hat{\pi} \to |\K \pi|\otimes |\K \pi|$.
For any $\gamma\in \pi$, we have
\begin{equation}
\label{eq:d^+alt2}
\delta^+(|\gamma |)=
-\alt (1\otimes |\ |)\mu(\gamma)+|\gamma| \wedge {\bf 1}.
\end{equation}
Its $\K$-linear extension
$\delta^+: |\K \pi| \cong \K \hat{\pi} \to |\K \pi|\otimes |\K \pi|$
is a lift of the Turaev cobracket in the sense that
$\varpi^{\otimes 2}\circ \delta^+=\delta \circ \varpi$, where
$\varpi:|\K \pi|\to |\K \pi|/\K {\bf 1}$ is the natural projection.
\end{prop}

\begin{proof}
To prove the first statement, we need to show
$\delta^+(|i(\alpha\beta)|^+)=\delta^+(|i(\beta\alpha)|^+)$
for any $\alpha,\beta\in \pi$.
We write $\mu(\alpha)=|\alpha'|\otimes \alpha''$ and
$\mu(\beta)=|\beta'|\otimes \beta''$.
Then, by (\ref{eq:d^+alt}) and Proposition \ref{prop:prod2},
\begin{align*}
\delta^+(|i(\alpha\beta)|^+) =&
-\alt (1\otimes |\ |)\mu(i(\alpha\beta))+|i(\alpha\beta)|\wedge {\bf 1} \\
=& -\alt (1\otimes |\ |)(\mu(\alpha)(1\otimes \beta)
+(1\otimes \alpha)\mu(\beta)+(|\ |\otimes 1)\kappa(\alpha,\beta))
+|\alpha\beta|\wedge {\bf 1} \\
=& -\alt (|\alpha'|\otimes |\alpha''\beta|+|\beta'|\otimes |\alpha\beta''|
+(|\ |\otimes |\ |)\kappa(\alpha,\beta))
+|\alpha\beta|\wedge {\bf 1} \\
=& |\alpha''\beta|\wedge |\alpha'|+|\alpha\beta''|\wedge |\beta'|
+|\alpha\beta|\wedge {\bf 1}-\alt (|\ |\otimes |\ |)\kappa(\alpha,\beta).
\end{align*}
Also, by (\ref{eq:k(y,x)}), 
\[
\alt (|\ |\otimes |\ |)\kappa(\beta,\alpha)=
-\alt (|\ |\otimes |\ |)(\kappa(\alpha,\beta)^{\circ})
=\alt (|\ |\otimes |\ |)\kappa(\alpha,\beta).
\]
These two computations prove
$\delta^+(|i(\alpha\beta)|^+)=\delta^+(|i(\beta\alpha)|^+)$,
as required.
The formula (\ref{eq:d^+alt2}) follows from (\ref{eq:d^+alt}).
Finally, by (\ref{eq:d^+alt2}), $\varpi^{\otimes 2}\delta^+(|\gamma|)
=-\alt (1\otimes |\ |)\mu(\gamma)=\delta \varpi (|\gamma|)$,
where the last equality follows from (\ref{eq:mu(g)}) and (\ref{eq:d^+}) (the latter being interpreted as the defining formula for $\delta$).
\end{proof}

\begin{rem}
In fact, the $\K$-vector space $|\K \pi|$ is a Lie bialgebra with respect to the Goldman bracket and the map $\delta^+$.
\end{rem}

By interchanging the roles of $*$ and $\bullet$, we obtain
an operation $\mu_{*\bullet}: \K \pi \to \K \pi \otimes |\K \pi|$ which will be used in the next section.
This map satisfies the product formula
\begin{equation}
\label{eq:prod*O}
\mu_{*\bullet}(xy)=\mu_{*\bullet}(x)(y\otimes 1)+
(x\otimes 1)\mu_{*\bullet}(y)+(1\otimes |\ |)\kappa(y,x)
\end{equation}
for any $x,y\in \K \pi$, and we have $\mu_{*\bullet}(\gamma_i^{-1})=0$ for all $i$.
Moreover, these two properties characterise $\mu_{*\bullet}$
(c.f. Proposition \ref{prop:prod2}).

One can show that
$(|\ |\otimes 1)(\mu_{\bullet *}(\gamma)^{\circ})
+|\gamma|\wedge {\bf 1}
=-(|\ |\otimes 1)\mu_{* \bullet}(\gamma)$
for any $\gamma\in \pi$. Therefore,
(\ref{eq:d^+alt2}) is rephrased as
\begin{equation}
\label{eq:O**O}
\delta^+(|\gamma|)=-(1\otimes |\ |)\mu_{\bullet *}(\gamma)
-(|\ |\otimes 1)\mu_{*\bullet}(\gamma).
\end{equation}

\if
By interchanging the roles of $*$ and $\bullet$, we introduce
a group $\pi_{*\bullet}^+$ and an operation $\mu_{*\bullet}$ which will be used in the next section.
Let $\pi_{*\bullet}^+$ be the set of regular homotopy classes of immersions $\gamma:([0,1],0,1)\to (\Sigma,*,\bullet)$ such that
$\dot{\gamma}(0)=\xi_*$ and $\dot{\gamma}(1)=-\xi_{\bullet}$.
For $\alpha, \beta \in \pi_{*\bullet}^+$,
their product $\alpha \beta$ is defined to be the insertion of a \emph{negative} monogon to $\alpha\nu\beta$.
The map $\pi_{*\bullet}^+\to \pi, \gamma\mapsto \gamma\nu$
is a surjective group homomorphism.
We have the rotation number function
$\rot_{*\bullet}:\pi^+_{*\bullet}\to (1/2)+\mathbb{Z}$, and
this gives a section $i_{*\bullet}:\pi\to \pi_{*\bullet}^+$
defined by the condition that
$\rot_{*\bullet}(i_{*\bullet}(\gamma))=1/2$ for any $\gamma\in \pi$.
The map $\pi_{*\bullet}^+ \to \pi=\pi_{\bullet *}^+, \gamma\mapsto \overline{\gamma}$ is a bijective anti-homomorphism;
for any $\alpha,\beta\in \pi_{*\bullet}^+$, we have
$\overline{\alpha\beta}=\overline{\beta} \overline{\alpha}$.

Introduce the $\K$-linear map
$\mu_{*\bullet}: \K \pi_{*\bullet}^+ \to \K \pi \otimes \K \hat{\pi}$
by setting
\[
\mu_{*\bullet}(\gamma):=
-(\iota\otimes \iota)\mu_{\bullet*}(\overline{\gamma})^{\circ}
\]
for $\gamma\in \pi^+_{*\bullet}$.
Here, the antipode of $\K\pi$ induces a $\K$-linear involution $\iota:\K \hat{\pi}\to \K\hat{\pi}$.
Using (\ref{eq:k(y,x)}), (\ref{eq:k(ix,iy)}), and Proposition \ref{prop:prod}, we can show that for any $\alpha,\beta\in  \pi_{*\bullet}^+$,
\[
\mu_{*\bullet}(\alpha\beta)=
\mu_{*\bullet}(\alpha)(\beta\nu \otimes 1)+
(\alpha\nu\otimes 1)\mu_{*\bullet}(\beta)
+(1\otimes |\ |)\kappa(\beta\nu,\alpha\nu).
\]
Composing $\mu_{*\bullet}$ with the section $i_{*\bullet}$,
we obtain a $\K$-linear map
$\mu_{*\bullet}: \K \pi \to \K \pi \otimes \K \hat{\pi}$
(using the same letter).
This map satisfies
\begin{equation}
\label{eq:prod*O}
\mu_{*\bullet}(xy)=\mu_{*\bullet}(x)(y\otimes 1)+
(x\otimes 1)\mu_{*\bullet}(y)+(1\otimes |\ |)\kappa(y,x)
\end{equation}
for any $x,y\in \K \pi$.
We can show that
$(|\ |\otimes 1)(\mu_{\bullet *}(i_{\bullet *}\gamma)^{\circ})
+|\gamma|\wedge {\bf 1}
=-(|\ |\otimes 1)\mu_{* \bullet}(i_{*\bullet}\gamma)$
for any $\gamma\in \pi$. Therefore,
(\ref{eq:d^+alt2}) is rephrased as
\begin{equation}
\label{eq:O**O}
\delta^+(|\gamma|)=-(1\otimes |\ |)\mu_{\bullet *}(\gamma)
-(|\ |\otimes 1)\mu_{*\bullet}(\gamma).
\end{equation}
\fi

\subsection{Completions}
\label{rem:comp}
Let $I\pi:=\ker (\varepsilon)$ be the augmentation ideal of $\K\pi$.
The powers $(I\pi)^p$, $p\ge 0$, give a descending filtration
of two-sideded ideals of $\K\pi$.
The projective limit
$\widehat{\K\pi}:=\varprojlim_{p\to \infty} \K\pi/(I\pi)^p$ naturally
has the structure of a complete Hopf algebra.
Also, put $\widehat{|\K \pi|}:=\varprojlim_{p\to \infty} |\K \pi|/|(I\pi)^p|$.
Since $\pi$ is a free group of finite rank,
the natural map $\K \pi \to \widehat{\K\pi}$ is injective and
the image is dense with respect to the filtration
$\widehat{I\pi}^p:=\ker (\widehat{\K \pi}\to \K\pi/ (I\pi)^p)$, $p\ge 0$.
The natural map $|\K \pi|\to \widehat{|\K \pi|}$ has a similar property.

As  shown in \cite{KK15} \S 4 and \cite{KK16} \S 4,
the operations $\kappa$, the Goldman bracket, $\mu$, and $\delta^+$
on $\K \pi$ and $\K \hat{\pi}$ have
natural extensions to the completions $\widehat{\K\pi}$ and $\widehat{|\K \pi|}$.
Therefore, we have continuous maps
$\kappa: \widehat{\K\pi} \widehat{\otimes} \widehat{\K\pi}
\to \widehat{\K\pi} \widehat{\otimes} \widehat{\K\pi}$,
$\{-,-\}: \widehat{|\K \pi|} \widehat{\otimes} \widehat{|\K \pi|} \to \widehat{|\K \pi|} \widehat{\otimes} \widehat{|\K \pi|}$,
$\delta^+:\widehat{|\K \pi|} \to \widehat{|\K \pi|} \widehat{\otimes} \widehat{|\K \pi|}$,
and
$\mu:\widehat{\K \pi}\to \widehat{|\K \pi|} \widehat{\otimes} \widehat{\K \pi}$.

\section{Expansions and transfer of structures}
\label{sec:mult}

In this section, we introduce a notion of expansions which allow to transfer the operations ({\em e.g.} the double bracket and cobracket) from the group algebra $\mathbb{K}\pi$ to the free associative algebra $A$.

Keep the notation in \S \ref{subsec:g0} and identify the homology class of $\gamma_i$ with the $i$th generator $x_i$ of $A=A_n$.
Then $H_1(\Sigma;\K)$ is isomorphic to the degree $1$ part of $A$. Recall the notation $x_0=-\sum_{i=1}^n x_i\in H_1(\Sigma;\K)\subset A$.

\begin{dfn}
A \emph{group-like} expansion of $\pi$ is a map $\theta: \pi\to A$ such that
\begin{enumerate}
\item[(i)] $\theta(\alpha\beta)=\theta(\alpha)\theta(\beta)$ for any $\alpha,\beta\in \pi$,
\item[(ii)] for any $\alpha\in \pi$,
$\theta(\alpha)=1+[\alpha]+(\text{higher terms})$,
where $[\alpha]\in H_1(\Sigma;\K)$ is the homology class of $\alpha$,
\item[(iii)] for any $\alpha\in \pi$, $\theta(\alpha)\in \exp(L)$.
\end{enumerate}
A group-like expansion $\theta$ is called \emph{tangential} if
\begin{enumerate}
\item[(iv)] for any $1\le i\le n$,
there exists an element $g_i\in L$ such that
$\theta(\gamma_i)=e^{g_i}e^{x_i}e^{-g_i}$.
\end{enumerate}
A tangential expansion is called \emph{special} if
\begin{enumerate}
\item[(v)] $\theta(\gamma_1\cdots \gamma_n)=\exp(-x_0).$
\end{enumerate}
\end{dfn}

\begin{exple}
There is a distinguished tangential expansion $\theta^{\rm exp}$ defined on generators
$\theta^{\rm exp}(\gamma_i)=e^{x_i}$. This expansion is not special since $\theta^{\exp}(\gamma_1\cdots \gamma_n)=
e^{x_1}\cdots e^{x_n} \neq e^{-x_0}$.
\end{exple}

\begin{rem}
The notion of a group-like expansion and that of a special expansion were introduced by Massuyeau in \cite{Mas12} and \cite{Mas15}, respectively.
The first example of special expansions was given by \cite{HM} by using the Kontsevich integral.
Special expansions appear implicitly in \cite{AET10} \S 4.1.
\end{rem}

Any group-like expansion $\theta$ induces an injective map $\theta: \K \pi \to A$ of Hopf algebras and an isomorphism
$\theta:\widehat{\K \pi} \xrightarrow{\cong} A$ of complete Hopf algebras (\cite{Mas12} Proposition 2.10).
It also induces an isomorphism $\theta:\widehat{|\K \pi|} \xrightarrow{\cong} |A|$ of filtered $\K$-vector spaces (\cite{KK16} Proposition 7.1).
By the discussion in \S \ref{rem:comp}, there exist unique  continuous maps
$\kappa_\theta: A\otimes A \to A \otimes A, \mu_{\theta}:A\to |A|\otimes A$ and
$\delta^+_{\theta}:|A|\to |A|\otimes |A|$ such that
the following diagrams are commutative:
\[
\begin{CD}
\K \pi \otimes \K \pi @> \kappa >> \K \pi \otimes \K \pi \\
@V \theta\otimes \theta VV @VV \theta \otimes \theta V \\
A\otimes A @> \kappa_\theta >> A\otimes A,
\end{CD}
\]

\[
\begin{CD}
\K \pi @> \mu >> |\K \pi| \otimes \K \pi \\
@V \theta VV @VV \theta \otimes \theta V \\
A @> \mu_{\theta} >> |A| \otimes A,
\end{CD}
\hspace{5em}
\begin{CD}
|\K \pi| @> \delta^+ >> |\K \pi| \otimes |\K \pi| \\
@V \theta VV @VV \theta \otimes \theta V \\
|A| @> \delta^+_{\theta} >> |A| \otimes |A|.
\end{CD}
\]

We will first study the maps $\kappa_\theta, \mu_\theta$ and $\delta^+_\theta$ for the expansion $\theta^{\rm exp}$. In order to proceed, we need the following technical lemma:

\begin{lem}
\label{lem:p(e^x)}
Let $N$ be an $A$-bimodule and let $\phi:A\to N$ be a $\K$-linear continuous map such that $\phi(ab)=\phi(a)b+a\phi(b)$ for any $a,b\in A$.
Then, for any $x\in A_{\ge 1}$,
\[
\phi(e^x)=e^x\left( \frac{1-e^{-\ad_x}}{\ad_x} \phi(x) \right).
\]
\end{lem}

\begin{proof}
Note that the expression $I(s)=\phi(e^{sx})$ satisfies the ordinary differential equation
$$
\frac{dI(s)}{ds} = \phi(x e^{sx})=\phi(x)e^{sx} + x\phi(e^{sx})=
\phi(x)e^{sx} + x I(s)
$$
with initial condition $I(0)=0$ (since $\phi(1)=0$). It is easy to see that the function
$$
e^{sx} \left( \frac{1-e^{-s\, {\rm ad}_x}}{{\rm ad}_x} \phi(x) \right)
$$
satisfies both the differential equation and the initial condition. Hence, by uniqueness of solutions of ordinary differential equation, the expression above gives a formula for $I(s)$. Putting $s=1$ yields the desired result.
\end{proof}

\begin{prop}
\label{prop:mult}

The double bracket $\Pi_{\rm mult}:=\kappa_{\theta^{\rm exp}}$  is tangential, and it is given by the following formula:
\begin{equation}   \label{eq:Pimult}
\Pi_{\mult} =\sum_{i=1}^n \left| \pa_i \otimes
\left( \frac{1}{1-e^{-\ad_{x_i}}} \right)
\ad^2_{x_i} \pa_i \right|
-\sum_{1\le j<i\le n} | \ad_{x_i}\pa_i \otimes \ad_{x_j}\pa_j |.
\end{equation}
\end{prop}

\begin{rem}
The expression $\Pi_\mult$ is well known in quasi-Poisson geometry. The first term defines a canonical (non skew-symmetric) bracket on a connected Lie group $G$ with quadratic Lie algebra, see Example 1 in \cite{Severa_left}. The second term is the so-called fusion term, see Section 4 in \cite{Severa_left}.
\end{rem}

\begin{proof}
It suffices to show that
$(\theta^{\exp}\otimes \theta^{\exp})\kappa(\gamma_i,\gamma_j)
=\Pi_{\mult}(e^{x_i},e^{x_j})$ for any $1\le i,j\le n$.
By (\ref{eq:ph1ph2}),
\[
\Pi_{\mult}(x_i,-)=
\left( \frac{1}{1-e^{-\ad_{x_i}}} \right)
\ad_{x_i}^2 \pa_i
+\sum_{k<i} \ad_{x_i}(\ad_{x_k}\pa_k).
\]
This proves that $\Pi_{\mult}$ is tangential.
By Lemma \ref{lem:p(e^x)}, we compute
\begin{align}
& \Pi_{\mult}(e^{x_i},-)=
e^{x_i} \left(
\frac{1-e^{-\ad_{x_i}}}{\ad_{x_i}} \Pi_{\mult}(x_i,-)
\right) \nonumber \\
=& e^{x_i} (\ad_{x_i}\pa_i)+\sum_{k<i} \left(
e^{x_i}(\ad_{x_k}\pa_k)-(\ad_{x_k}\pa_k)e^{x_i} \right). \label{eq:P(e^x,-)}
\end{align}
From this equation, we obtain
\begin{align}
&\Pi_{\mult}(e^{x_i},e^{x_i})=
1\otimes e^{x_i}e^{x_i}-e^{x_i}\otimes e^{x_i} \quad \text{for any $i$}, \nonumber \\
&\Pi_{\mult}(e^{x_i},e^{x_j})=0 \quad \text{if $i<j$}, \nonumber \\
&\Pi_{\mult}(e^{x_i},e^{x_j})=
1\otimes e^{x_i}e^{x_j}+e^{x_j}e^{x_i}\otimes 1
-e^{x_i}\otimes e^{x_j}-e^{x_j}\otimes e^{x_i}
\quad \text{if $i>j$.} \label{eq:Pi(ex,ex)}
\end{align}
Comparing this with Proposition \ref{prop:k(ga,ga)}, we obtain
the result. 
\end{proof}

\begin{prop}
\label{prop:mult_div}
We have
$$
\mu_{\theta^{\rm exp}} = \tdiv^T_{\rm mult}, \hskip 0.3cm
(\mu_{*\bullet})_{\theta^{\rm exp}} = \underline{\tdiv}^T_{\mult}, \hskip 0.3cm
\delta^+_{\theta^{\rm exp}} = - \tdiv_\mult.
$$
\end{prop}

\begin{proof}
For the computation of $\mu_{\theta^{\rm exp}}$,
Proposition \ref{prop:prod2} and equation (\ref{eq:tdivab1}) show that $\mu$ and $\tdiv^T_{\rm mult}$ satisfy the same product formulas, and Proposition \ref{prop:mult}
shows that $\theta^{\rm exp}$ intertwines the double brackets $\kappa$ and $\Pi_{\rm mult}$.
Hence, it remains to show that $\theta^{\rm exp}$ intertwines $\mu$ and $\tdiv^T_{\rm mult}$
on some set of generators  of $\mathbb{K}\pi$.
By Proposition \ref{prop:prod2},  $\mu(\gamma_i)=0$ for all $i$.
The fact that $\tdiv^T_{\mult}(e^{x_i})=0$ can be checked directly from (\ref{eq:P(e^x,-)}) using (\ref{eq:tdiv^T}) and the fact that $\pa_k (e^{x_i})=0$ for $k<i$.

The formula for $(\mu_{*\bullet})_{\theta^{\rm exp}}$ can be proved similarly:
we use the fact that
$\mu_{* \bullet}(\gamma_i^{-1})=0$ for any $i$, and the computation
of $\Pi_{\mult}(e^{-x_i},e^{-x_j})$ obtained from (\ref{eq:Pi(ex,ex)}).
The formula for $\delta^+_{\theta^{\rm exp}}$ follows from 
equations (\ref{eq:tdiv(a)}) and (\ref{eq:O**O}).
\end{proof}

\begin{rem}
Explicit expressions for  $\mu_{\theta^{\rm exp}}$ and $\delta^+_{\theta^{\rm exp}}$ were obtained in \cite{Ka16}.
\end{rem}

For arbitrary special expansions,
Massuyeau and Turaev proved the following theorem:

\begin{thm}[\cite{MTpre}, see also \cite{Mas15} Theorem 5.2]
\label{thm:MT}
For any special expansion $\theta$, we have
$$
\kappa_\theta=\Pi_{\rm add}:= \Pi_{\rm KKS} + \Pi_s,
$$
where 
$$
s(z)= \frac{1}{z}-\frac{1}{1-e^{-z}}.
$$
\end{thm}

\begin{rem}
Actually, Massuyeau and Turaev work mainly with the homotopy intersection form $\eta:\K \pi\otimes \K \pi \to \K \pi$.
Since $\kappa$ can be recovered from $\eta$ (\cite{KK16} Proposition 4.7) and vice versa, it is easy to translate their result into the form presented in  Theorem \ref{thm:MT}.
\end{rem}

\begin{rem}
Theorem \ref{thm:MT} admits an interpretation in terms of the exponentiation construction in quasi-Poisson geometry, see Section 7 in \cite{AKM}.  The KKS double bracket is a double Poisson bracket, and adding the term $\Pi_s$ turns it into a quasi-Poisson double bracket in the sense of van den Bergh. 
\end{rem}

\begin{rem}
Note that the function $s(z)$ admits the following presentation in terms of Bernoulli numbers:
$$
s(z)=-\frac{1}{2} - \sum_{k=1}^\infty \frac{B_{2k}}{(2k)!} \, z^{2k-1} .
$$
It defines a solution of the classical dynamical Yang-Baxter equation (CDYBE).
\end{rem}

As a consequence of Theorem \ref{thm:MT}, for a special expansion $\theta$, we have the following commutative diagram:
\begin{equation}
\label{eq:GoKKS}
\begin{CD}
|\K \pi| \otimes |\K \pi| @> \{-,-\} >> |\K \pi| \\
@V \theta\otimes  \theta VV @VV \theta V \\
|A|\otimes |A| @> \{-,-\}_{\KKS} >> |A|.
\end{CD}
\end{equation}
This is obtained independently by Massuyeau and Turaev \cite{MTpre}
and by the second and third authors \cite{KK16}. Note that $\Pi_s$ does not contribute in the Lie bracket on $|A|$
(Proposition \ref{prop:add'}).

\begin{rem}  \label{rem:mult_add}
By combining Theorem \ref{thm:MT} and Proposition \ref{prop:mult} we obtain the following result: let $F \in \taut(L)$ be such that $F(x_1 + \dots + x_n)=\log(e^{x_1} \dots e^{x_n})$. Then, 
\begin{equation} \label{eq:mult_add}
(F^{-1} \otimes F^{-1}) \Pi_\mult (F \otimes F) = \Pi_\add.
\end{equation}

In \cite{Florian}, one gives an algebraic proof of this result and shows that its inverse also holds true. That is, if $F \in \taut(L)$ verifies \eqref{eq:mult_add} then $F(x_1+ \dots +x_n)=\log(e^{x_1} \dots e^{x_n})$. 


One can pose the following natural question: chracterize elements $F \in \taut(L)$ which induce Lie isomorphisms $(|A|, \{ \cdot, \cdot\}_{\rm mult}) \to (|A|, \{ \cdot, \cdot\}_{\rm KKS})$. Conjecturally, they are given by compositions $G \circ F$, where $G$ is an inner automorphism and $F(x_1 \dots x_n) =\log(e^{x_1} \dots e^{x_n})$.

\end{rem}

\begin{rem}   \label{rem:center_theta}
Let $F \in \taut(L)$ and $\theta_F=F^{-1} \circ \theta^{\rm exp}$. Then, the map $F$ induces an isomorphism of Lie algebras 
$$
(|A|, \{-,-\}_\mult) \cong (|A|, \{-,-\}_\theta).
$$
In particular, this isomorphism restricts to an isomorphism of the centers. Since we know the center of the Lie algbera $(|A|, \{-,-\}_{\rm KKS})$, we conclude that the center of the Lie algebra $(|A|, \{-,-\}_\theta)$ is spanned by the elements $|(F^{-1}(x_i))^k|=|x_i^k|$ for $i=1,\dots, n$ and by
$|(F^{-1}\log(e^{x_1}\dots e^{x_n}))^k|$. 
\end{rem}

\section{The Kashiwara-Vergne problem and homomorphic expansions}
\label{sec:KVA}

In this section, we discuss the connection between the Kashiwara-Vergne (KV) problem in Lie theory and the transfered structures $\mu_\theta$ and $\delta^+_\theta$. 


\subsection{KV problems for surfaces of genus zero}

We start with the formulation of the KV problem:
\vskip 0.2cm

\noindent \textbf{Kashiwara-Vergne problem:}
Find an element $F\in \taut(L_2)$
satisfying the conditions
\[
\tag{KVI}
F(x_1+x_2)=\log(e^{x_1}e^{x_2}),
\]
\[
\tag{KVII}
\exists\,  h(z) \in \K[[z]] \,\,  {\rm such} \,\, {\rm that} \, j(F^{-1})=|h(x_1)+h(x_2)-h(x_1+x_2)|.
\]

\begin{rem}
The KV problem was originally discovered as a property of the Baker-Campbell-Hausdorff series which implies the Duflo theorem on the center of the universal enveloping algebra. The formulation above follows  \cite{AT12}. 
\end{rem}

Recall the following result:

\begin{thm}[\cite{AM06},\cite{AT12}]
The KV problem admits solutions.
\end{thm}

As will be shown in our main theorems, the Kashiwara-Vergne problem 
is closely related to the topology of a surface of genus $0$ 
with $3$ boundary components. 
Regarding the topology of a surface of genus $0$ 
with $n+1$ boundary components for $n\geq 2$, 
we introduce a generalization of the Kashiwara-Vergne problem.

\vskip 0.2cm

\par
\noindent 
\textbf{Kashiwara-Vergne problem of type $(0,n+1)$:}
An element $F \in \taut(L_n)$ is a solution to 
the KV problem of type $(0, n+1)$ if 
\[
\tag{$\mathrm{KV}^{(0,n+1)}\mathrm{I}$}
F(x_1+x_2+\dots+x_n)= \log(e^{x_1}e^{x_2}\cdots e^{x_n}),
\]
\[
\tag{$\mathrm{KV}^{(0,n+1)}\mathrm{II}$}
\exists h(z)\in \K[[z]] \,\, {\rm such} \,\, {\rm that} \,\, j(F^{-1})=\left|\sum^n_{i=1}h(x_i)-h\left(\sum^n_{i=1}x_i\right)\right|.
\]

\par
The conditions ($\mathrm{KV}^{(0,3)}\mathrm{I}$) 
and ($\mathrm{KV}^{(0,3)}\mathrm{II}$) coincide with  
the conditions (KVI) and (KVII), respectively. 
The existence of a solution to the problem of type $(0,n+1)$ follows 
from that to the original problem as follows.

Let $F\in \taut(L_2)$ and write $F=\exp(u)$ with $u=(u_1,u_2)\in \tder(L_2)$.
For $n\ge 2$, we introduce the following elements of $\tder(L_n)=L^{\oplus n}$:
\begin{align*}
u^{n-1,n}:=& (0,\ldots,0,u_1(x_{n-1},x_n),u_2(x_{n-1},x_n)), \\
u^{n-2,(n-1)n}:=&
(0,\ldots,0,u_1(x_{n-2},x_{n-1}+x_n),
u_2(x_{n-2},x_{n-1}+x_n),u_2(x_{n-2},x_{n-1}+x_n)), \\
\vdots &  \\
u^{2,3\cdots n} :=&
(0,u_1(x_2,x_3+\cdots+x_n),u_2(x_2,x_3+\cdots+x_n),\ldots,
u_2(x_2,x_3+\cdots+x_n)), \\
u^{1,2\cdots n}:=&
(u_1(x_1,x_2+\cdots+x_n),u_2(x_1,x_2+\cdots+x_n),\ldots,
u_2(x_1,x_2+\cdots+x_n)).
\end{align*}
Then, set
\[
F^{(n)}:=F^{n-1,n}\circ F^{n-2,(n-1)n}\circ \cdots \circ
F^{2,3\cdots n} \circ F^{1,2\cdots n} \in \taut(L_n),
\]
where $F^{n-1, n}=\exp(u^{n-1, n})$ {\em etc.}
For more details about this operadic notation, see \cite{AET10}, \cite{AT12}.
For our purpose, the \emph{naturality} of this construction is important.
For example, consider the map $A_2\to A_n,a\mapsto a^{1,2\cdots n}=a(x_1,x_2+\cdots +x_n)$ which maps $x_1$ to $x_1$ and $x_2$ to $x_2+\cdots+x_n$.
Then, $F(a)^{1,2\cdots n}=F^{1,2\cdots n}(a^{1,2\cdots n})$ for any $a\in 
A_2$.
The divergence map is natural in this sense (\cite{AT12} \S 3), and so is its integrating cocycle $j$.
For example, $j(F)^{1,2\cdots n}=j(F^{1,2\cdots n})$ for any $F\in \taut(L_2)$.

\begin{lem}
\label{lem:F(-x_0)}
Let $F \in \taut(L_2)$ be a solution of
{\rm (KVI)} and {\rm (KVII)}. Then,
$F^{(n)}\in \taut(L_n)$ is a solution 
to the Kashiwara-Vergne problem of type $(0, n+1)$.
\end{lem}

\begin{proof}
We first show that $F^{(n)}$ satisfies ($\mathrm{KV}^{(0,n+1)}\mathrm{I}$).
We have
$F^{i,(i+1)\cdots n}(x_i+x_{i+1}+\cdots+x_n)
=\bch(x_i,x_{i+1}+\cdots+x_n)$ for $i=1,\ldots,n-1$.
Successive application of this equality yields
\begin{align*}
F^{(n)}(x_1+\cdots+x_n)
=& \bch(x_1,\bch(x_2,\bch(x_3,\cdots,\bch(x_{n-1},x_n)\cdots))) \\
=& \log(e^{x_1}e^{x_2}\cdots e^{x_n}).
\end{align*}

To prove that $F^{(n)}$ satisfies ($\mathrm{KV}^{(0,n+1)}\mathrm{II}$),
put $G_i:=(F^{i,(i+1)\cdots n})^{-1}$ for $i=1,\ldots,n-1$.
Then $j(G_i)=|h(x_i)+h(x_{i+1}+\cdots+x_n)-h(x_i+\cdots+x_n)|$
and $G_i$ acts on the last $n-i$ generators $x_{i+1},\ldots,x_n$
as an inner automorphism; there exists $g_i\in \exp(L_n)$ such
that $G_i(x_k)=g_i^{-1}x_kg_i$ for $i<k\le n$.
Therefore, $G_1\cdots G_{i-1}$ acts trivially on $j(G_i)$.
By (\ref{eq:jco}), we compute
\begin{align*}
j((F^{(n)})^{-1})=&
j(G_1G_2\cdots G_{n-1}) \\
=& j(G_1)+G_1 \cdot j(G_2)+
\cdots +G_1G_2\cdots G_{n-2}\cdot j(G_{n-1}) \\
=& \sum_{i=1}^{n-1} |h(x_i)+h(x_{i+1}+\cdots+x_n)
-h(x_i+\cdots+x_n)| \\
=& \left| \sum_{i=1}^{n} h(x_i)-h(x_1+\cdots+x_n) \right|.
\end{align*}
This completes the proof.
\end{proof}

\subsection{Homomorphic expansions}
In this section, we define the notion of homomorphic expansion 
and show that expansions defined by solutions of the KV problem are in fact homomorphic. 

\begin{lem}
\label{lem:Fspe}
Assume that $F \in \taut(L_n)$ satisfies 
$(\mathrm{KV}^{(0,n+1)}\mathrm{I})$. Then,
the map $\theta_F:=F^{-1}\circ \theta^{\exp}$ is a special
expansion.
\end{lem}

\begin{proof}
Referring to the definition of special expansions,
the condition (i) is satisfied since $F^{-1}$ and $\theta^{\exp}$
are algebra homomorphisms.
Since $F^{-1}\in \taut(L_n)$, for each $i$ 
there exists an element $g_i\in L$ such that
$F^{-1}(x_i)=e^{g_i}x_ie^{-g_i}$.
Therefore,
$\theta_F(\gamma_i)=F^{-1}(e^{x_i})=e^{g_i}e^{x_i}e^{-g_i}$,
proving (iv).
The condition (v) follows from the condition 
($\mathrm{KV}^{(0,n+1)}\mathrm{I}$):
$\theta_F(\gamma_1\cdots \gamma_n)
=F^{-1}(e^{x_1}\cdots e^{x_n})=\exp(-x_0)$.
The conditions (ii) and (iii) follow from (i) and (iv).
\end{proof}

We introduce a class of expansions which transform the Goldman bracket and the Turaev cobracket $\delta^+$ in a nice way:

\begin{dfn}
A special expansion $\theta: \pi \to A$ is called \emph{homomorphic}  if $\delta^+_\theta = \delta^{\rm alg}$.
\end{dfn}

Our terminology ``homomorphic'' follows the work of Bar-Natan and Dancso \cite{B-ND}.
Homomorphic expansions are closely related to solutions of the generalized Kashiwara-Vergne problem:

\begin{thm}
\label{thm:main}
Let $F\in \taut(L_n)$ be a solution 
to the Kashiwara-Vergne problem of type $(0,n+1)$.

\begin{enumerate}
\item[(i)]
Let us write $\tde(s(-x_0))=s'\otimes s''$ where $s(z)$ is the function in Theorem \ref{thm:MT}, let $g(z):=\dot{h}(z)\in \K[[z]]$ be the derivative of $h(z)$ and write
$\tde(g(-x_0))=g'\otimes g''\in A\otimes A$.
Then, for any $a\in A$,
\begin{equation}
\label{eq:muF}
\mu_{\theta_F}(a)=\mu^{\alg}(a)+|s''|\otimes as'-|s''a|\otimes s'+|g'|\otimes [a,g''].
\end{equation}
\item[(ii)]
We have $\delta^+_{\theta_F}=\delta^{\alg}$.
\end{enumerate}
\end{thm}

\begin{rem} \label{rem:Mas}
Massuyeau \cite{Mas15} proved that for any special expansion $\theta$ arising from the Kontsevich integral, $\delta^+_{\theta}=\delta^{\alg}$.
He also considered an operation which is closely related to our $\mu$, and showed its tensorial description similar to (\ref{eq:muF}).
\end{rem}

\begin{proof}[Proof of Theorem \ref{thm:main}]
We split the proof into several steps:

\vskip 0.2cm

\emph{Step 1.} Let $F \in \taut(L_n)$ be a solution of the KV problem of type $(0, n+1)$.
Consider the expansion $\theta_F=F^{-1} \circ \theta^{\rm exp}$ and the corresponding map 
$\mu_F=\mu_{\theta_F}$. By Proposition \ref{prop:mult_div}, $\mu_{\theta^{\rm exp}} = \tdiv^T_{\Pi_{\rm mult}}$. Hence, 
$$
\begin{array}{lll}
\mu_F(a) & = &  (F^{-1} \otimes F^{-1})\cdot \tdiv^T(\{ F(a), -\}_\mult) \\
& = & F^* \tdiv^T(\{ a, -\}_\add) \\
& = & \tdiv^T(\{ a, -\}_\add) + (1 \otimes \{ a, -\}_\add) \tJ(F^{-1}) \\
& = & \tdiv^T_\add(a) + (1 \otimes \{ a, -\}_\add) \tJ(F^{-1}).
\end{array}
$$
Here in the second line we used equation \eqref{eq:mult_add}, and in the third line 
the transformation property \eqref{eq:F*tdivT} of the map $\tdiv^T$. We now discuss the two terms on the right hand side separately.

\vskip 0.2cm

\emph{Step 2.} By Proposition  \ref{prop:add'_div}, we have
\begin{equation}
\label{eq:tdivKKS}
\tdiv^T_{\add}(a)=\tdiv^T_{\KKS}(a)+\tdiv^T_s(a)
=\mu^{\alg}(a)+|s''|\otimes as'-|s''a|\otimes s'.
\end{equation}
Note that an explicit formula for $\mu^{\rm alg}=\tdiv^T_{\rm KKS}$ is given in Proposition \ref{prop:KKS_div}.

\vskip 0.2cm

\emph{Step 3.} We now turn to the expression $(1 \otimes \{ a, -\}_\add) \tJ(F^{-1})$. It is convenient to use the Sweedler notation for $\tJ(F^{-1})=\tde(j(F^{-1})) = |f'| \otimes |f''|.$ Again, we represent the double bracket as a sum of two terms, $\Pi_\add=\Pi_{\rm KKS} + \Pi_s$. The term coming from $\Pi_s$ is equal to zero,
$$
(1 \otimes \{ a, -\}_s) \tJ(F^{-1}) = |f'| \otimes \{ a, |f''|\}_s =0
$$
since $\Pi_s$ vanishes on $|A|$ by Proposition \ref{prop:add'}.

\vskip 0.2cm

\emph{Step 4.} We now turn to the analysis of the term $(1 \otimes \{ a, -\}_{\rm KKS}) \tJ(F^{-1})$. 
By the second KV equation
($\mathrm{KV}^{(0,n+1)}\mathrm{II}$), we have $j(F^{-1})=|\sum_{i=1}^n h(x_i) -h(-x_0)|$. Note that $\tde(|{x_i}^k|)$ belongs to the span of elements of the form $|{x_i}^l| \otimes |{x_i}^m|$. By Lemma \ref{prop:KKS}, $\{ a, |h(x_i)| \}_{\rm KKS}=0$ for all $a \in A$. Hence, 
$$
(1 \otimes \{ a, -\}_{\rm KKS}) \tde(\sum_i |h(x_i)|) =0.
$$

\emph{Step 5.} The remaining contribution is as follows:
$$
\begin{array}{lll}
- (1 \otimes \{ a, -\}_{\rm KKS}) \tde(|h(-x_0)|) & = &
- (1 \otimes \{ a, -\}_{\rm KKS}) |h(1 \otimes x_0 - x_0 \otimes 1)| \\
& = & ( | \ | \otimes 1) [1 \otimes a, \dot{h}(1 \otimes x_0 - x_0 \otimes 1)].
\end{array}
$$
In the third line, we used Lemma \ref{prop:KKS}. Applying the Sweedler notation
$$
\tde(g(-x_0))=\tde(\dot{h}(-x_0))=\dot{h}(1 \otimes x_0 - x_0 \otimes 1) = g' \otimes g'',
$$ 
we can rewrite the result as $|g'| \otimes [a, g'']$.  This concludes the proof of equation \eqref{eq:muF}.

\vskip 0.2cm

\emph{Step 6.}
From the defining formulas for $\mu^{\alg}$ and $\delta^{\alg}$,
we see that $\alt (1\otimes |\ |)\mu^{\alg}(a)=-\delta^{\alg}(|a|)$ for any $a\in A$.
Since the constant term of $s(z)$ is $-1/2$ and all the other terms
are of odd degree, we have
$s''\otimes s'= - 1\otimes 1-s'\otimes s''$.
Therefore, from (\ref{eq:muF}) we get
$\alt (1\otimes |\ |)\mu_{\theta_F}(a)=-\delta^{\alg}(|a|)+|a|\wedge {\bf 1}$ for any $a\in A$.
By (\ref{eq:d^+alt2}), we have
\[
(\theta_F\otimes \theta_F)\delta^+(|\gamma|)
=-\alt (1\otimes |\ |)\mu_{\theta_F}(\theta_F(\gamma))
+|\theta(\gamma)| \wedge {\bf 1}
=\delta^{\alg}(|\theta_F(\gamma)|)
\]
for any $\gamma\in \pi$.
Since the image $\theta_F(|\K \pi|)$ is dense in $|A|$,
we conclude $\delta^+_{\theta_F}=\delta^{\alg}$.

This completes the proof of Theorem \ref{thm:main}.
\end{proof}

\section{Topological interpretation of the Kashiwara-Vergne theory}
\label{sec:KVL}

Theorem \ref{thm:main} shows that the Kashiwara-Vergne equations for an element $F \in \taut(L_n)$ are sufficient to prove the equality $\delta^+_{\theta_F} = \delta^{\rm alg}$ for the expansion 
$\theta_F=F^{-1} \circ \theta^{\rm exp}$. In this section, our goal is to show that assuming the first KV equation, the second one becomes a necessary condition for  $\delta^+_{\theta_F} = \delta^{\rm alg}$. 

\subsection{Special derivations}

As a technical preparation, we discuss Lie algebras of special derivations and Kashiwara-Vergne Lie algebras.

\begin{dfn}
\label{dfn:spd}
An element $u\in \tder(A)$ is called \emph{special}
if $\rho(u)(x_0)=0$.
\end{dfn}

We denote by $\sder(A)$ the set of special elements in $\tder(A)$.
Namely,
\[
\sder(A)=\{ u=(u_1,\ldots,u_n)\in \tder(A) \mid \sum_i [x_i,u_i]=0 \}.
\]
This is a Lie subalgebra of $\tder(A)$.
Likewise, we define $\sder(L)$ to be the set of special elements in $\tder(L)$.

Consider the symmetrization map $N: |A| \to A$. For $a=z_1\cdots z_m$ with $z_j \in \{x_i\}_{i=1}^n$, we have
$$
N(|a|) = \sum_{j=1}^m \, z_j\cdots z_m z_1\cdots z_{j-1},
$$
and $N(|a|)=0$ for $a \in \mathbb{K}$. If $a$ is homogeneous of degree $m$, then $|N(|a|)| = m|a|$.

\begin{rem}
The symmetrization map $N$ coincides with the map $|A| \to A$ induced by the double derivation $\sum_i x_i \partial_i$.
\end{rem}

\begin{lem}
\label{lem:m-times}
For any $|a| \in |A|$, the derivation $\{ |a|, -\} \in \sder(A)$ is special, and it is given by
\[
\{|a|, - \}_{\KKS}=-(u_1,\ldots,u_n),
\]
where $u_i \in A$ are uniquely defined by formula $N(|a|)= \sum_{i=1}^n x_i u_i$.
Furthermore, the map $|a|\mapsto \{|a|,- \}_{\KKS}$ induces an isomorphism
\begin{equation} \label{eq:exacta}
|A|/\K {\bf 1}  \cong  \sder(A).
\end{equation}
\end{lem}

\begin{proof}
 Lemma \ref{prop:KKS} implies that $\{ |a|, x_0\}_{\rm KKS} = 0$ for all $a \in A$. Hence, $\{ |a|, -\}_{\rm KKS} \in \sder(A)$.
In order to prove the formula for the tangential derivation $\{ |a|, -\}_{\rm KKS}$, we may assume that $a=z_1\cdots z_m$ with $z_j\in \{x_i\}_{i=1}^n$.
Then the $i$th component of $\{|a|,-\}_{\KKS}$ is of the form
\[
-\sum_j \delta_{x_i,z_j} z_{j+1}\cdots z_mz_1\cdots z_{j-1}
\]
(see the proof of Proposition \ref{prop:KKS_div}),
and this is equal to $-u_i$.

Let $|a|\in |A|$ be a homogeneous element of degree $m\ge 1$ in the kernel of the map $|a| \mapsto \{ |a|, -\}_{\rm KKS}$.
The formula for  $\{|a|,-\}_{\KKS}$ shows that $|a|=0$.
Clearly, ${\bf 1}$ is in the kernel of the map $|A|\to \sder(A)$.
Hence, the map $|A|/\mathbb{K}{\bf 1} \to \sder(A)$ is injective.

To show the surjectivity,
let $v=(v_1,\ldots,v_n)$ be a special derivation and assume that $v_i$ is homogeneous of degree $m$ for any $i$.
Then, $a:=-(1/(m+1)) \sum_i x_i v_i$ is cyclically invariant since $\sum_i [x_i,v_i]=0$, and $N(|a|)=(m+1) a= -\sum_i x_i v_i$.
Therefore, $\{|a|,-\}_{\KKS}=v$.
\end{proof}

\subsection{The Lie algbera $\krv_n$}

\begin{dfn}[\cite{AT12} \S 4]
\label{dfn:KVL}
The \emph{Kashiwara-Vergne Lie algebra} $\mathrm{krv}_2$ is defined by
\[
\mathrm{krv}_2:=\{ u\in \sder(L_2) \mid
\exists h(z)\in \K[[z]], \div(u)=|h(x)+h(y)-h(x+y)| \}.
\]
\end{dfn}

Similarly, for any $n \geq 2$, we introduce a generalization of $\krv_2$, 
\[
\krv_n:=\{ u\in \sder(L) \mid
\exists h(z)\in \K[[z]], \, \div(u)=|\sum^n_{i=1}h(x_i)- h(-x_0)| \}.
\]

%

For our purposes, it is convenient to introduce a slightly different version of the Kashiwara-Vergne Lie algebra,
\[
\krv^{\rm KKS}_n=\{ u\in \sder(L) \mid \div(u) \in Z(|A|, \{ -, -\}_{\rm KKS})\} .
\]
Recall that the kernel of the map $\rho: \tder(L) \to \der(L)$ is spanned by the elements
$$
q_i =-\frac{1}{2}\{ |x_i^2|, - \}_{\rm KKS}= (0,\dots, 0, \overset{i}{\breve{x_i}}, 0, \dots, 0)\in \sder(L) \subset \tder(L).
$$
The relation between the two versions of the Kashiwara-Vergne Lie algebra is given by the following proposition:

\begin{prop} \label{thm:krvd}
\[
\krv^{\rm KKS}_n= \krv_n \oplus \left( \bigoplus_{i=1}^n \mathbb{K} q_i \right),
\]
\end{prop}

For the proof, we need the following lemma:

\begin{lem}
\label{lem:x=0}
Consider the substitution map
$\varpi_k: |A_n|\to |A_1|\cong \K [[z]], a\mapsto a|_{x_k=z,x_{k'}=0 (k'\neq k)}$,
given by $x_k\mapsto z$ and $x_{k'}\mapsto 0$ for $k'\neq k$.
Let $u=(u_1,\dots, u_n)\in \tder(L_n)$ and assume that
the degree $1$ part of $u_k$ is in $\bigoplus_{k'\neq k}\K x_{k'}$.
Then we have $\varpi_k(\div(u))=0$.
Furthermore, we have $\varpi_k(j(\exp(u)))=0$.
\end{lem}
\begin{proof}
If $v=[z_1,[z_2,\ldots,[z_{m-1},z_m]\cdots]]$,
where $z_j\in \{x_i\}_{i=1}^n$, is a Lie monomial of degree $m\ge 2$,
then $\varpi_k(v^k)=0$. 
By linearity, we see that $\varpi_k(v^k)=0$ if
the degree $1$ part of $v\in L$ is in $\bigoplus_{k'\neq k}\K x_{k'}$. 
Now, since $\varpi_k(x_{k'})=0$ for $k'\neq k$, we obtain
$$
\varpi_k(\div(u)) = \varpi_k|x_k(u_k)^k| + \sum_{k' \neq k}\varpi_k|x_{k'}(u_{k'})^{k'}| = |z\cdot 0| = 0.
$$
Finally, by (\ref{eq:jfdiv}), we obtain $j(\exp(u))|_{x_k=z,x_{k'}=0 (k' \neq k)}=0$.
\end{proof}

\begin{proof}[Proof of Proposition \ref{thm:krvd}]
For the inclusion $\krv_n \oplus (\bigoplus_{i=1}^n \mathbb{K} q_i) \subset \krv_n^{\rm KKS}$, we first consider the elements $q_i$. Since 
$$
\div(q_i)=|x_i| \in Z(|A|, \{ -, -\}_{\rm KKS}),
$$
we conclude that $q_i \in \krv_n^{\rm KKS}$. If $u\in \krv_n$, then $\div(u)=|\sum^n_{i=1}h(x_i)-h(-x_0)|$ 
for some $h(z)\in \K[[z]]$. In particular, $\div(u)$ is in the center 
with respect to the Lie bracket $\{-,-\}_{\KKS}$. Therefore, $\krv_n \subset \krv^{\rm KKS}_n$.

Note that expressions of the form $|\sum_{i=1}^n h(x_i) - h(-x_0)|$
contain no linear terms. Therefore, $\krv_n \cap \bigoplus_{i=1}^n \mathbb{K} q_i = \{ 0 \}$ and $\krv_n \oplus (\bigoplus_{i=1}^n \mathbb{K} q_i) \subset \krv_n^{\rm KKS}$, as required.

To prove the inclusion in the opposite direction, 
let $u\in \krv_n^{\rm KKS}$.
By Theorem \ref{thm:center},
there exist elements $h_i(z) \in z^2\K[[z]]$, $0 \le i\le n$, and
$c_i\in \K$, $0 \le i\le n$, such that
$$
\div(u) = |c_0 + \sum_{i=1}^n (c_ix_i + h_i(x_i))- h_0(-x_0)|.
$$
From the definition of the divergence, we have $c_0 = 0$. 
Hence, 
$$
\div(u-\sum^n_{i=1}c_iq_i) 
= |\sum^n_{i=1}h_i(x_i)-h_0(-x_0) |.
$$
By Lemma \ref{lem:x=0}, we obtain $h_i(z)=h_0(z)$.
Therefore, $u-\sum^n_{i=1}c_iq_i\in \krv_n$ and 
$u\in \krv_n\oplus (\bigoplus^n_{i=1}\K q_i)$.
\end{proof}

The Lie algebras $\sder(A)$, $\sder(L)$, $\krv_n$ and $\krv_n^{\rm KKS}$ integrate
to  the groups $\saut(A)$, $\saut(L)$,  $\KRV_n$ and $\KRV_n^{\rm KKS}$, respectively.
Clearly, $\KRV_n \subset \KRV_n^{\rm KKS} \subset \saut(L)$.
A more explict description of the group $\KRV_n$ is as follows:
\[
\KRV_n=\{ F\in \saut(L) \mid
\exists h(z)\in \K[[z]], j(F)=|\sum^n_{i=1}h(x_i) -h(-x_0)| \}.
\]
To see this, let $u\in \krv_n$ and $F=\exp(u)\in \KRV_n$.
We have $\div(u)=|\sum^n_{i=1}h(x_i)-h(-x_0) |$ for some $h(z)\in \K[[z]]$.
Since $u$ is a special derivation, $u\cdot \div(u)=u\cdot |\sum^n_{i=1}h(x_i)-h(-x_0)|=0$.
By (\ref{eq:jfdiv}), $j(F)=\div(u)+(1/2)u\cdot \div(u)+\cdots=\div(u)$.
The other inclusion can be proved similary (note that the operator $(e^{\ad_u}-1)/\ad_u$ is invertible).

The group $\KRV_n^{\rm KKS}$ admits a similar description:
\[
\KRV^{\rm KKS}_n=\{ F\in \saut(L) \mid
 j(F)\in Z(|A|, \{ -, -\}_{\rm KKS}) \}.
\]
As before, the fact that $j(F)\in Z(|A|, \{ -, -\}_{\rm KKS})$ is equivalent to the existence of $h(z) \in \mathbb{K}[[z]]$ such that
$$
j(F)=|\sum^n_{i=1}h(x_i) -h(-x_0)| \,\, {\rm mod} \,\, \bigoplus_{i=1}^n \mathbb{K} |x_i|.
$$

\subsection{Topological interpretation of the KV problem}

The following theorem gives a topological interpretation of the condition
($\mathrm{KV}^{(0,n+1)}\mathrm{II}$).

\begin{thm}\label{thm:main2}
Suppose that $F\in \taut(L)$ satisfies {\rm 
($\mathrm{KV}^{(0,n+1)}\mathrm{I}$)} and 
set $\theta_F:=F^{-1}\circ \theta^{\exp}$.
Then, $F$ satisfies $\delta^+_{\theta_F}=\delta^{\alg}$
if and only if 
$$
j(F^{-1}) \in Z(|A|, \{-, -\}_{\rm KKS}).
$$
Equivalently, there exists an element $G\in \ker(\rho)\subset \saut(L)$ such that $FG$ satisfies {\rm 
($\mathrm{KV}^{(0,n+1)}\mathrm{II}$)}.
\end{thm}

\begin{proof}
Let $F\in \taut(L)$ be a solution of  {\rm  ($\mathrm{KV}^{(0,n+1)}\mathrm{I}$)}. Consider the following diagram:
\[
\begin{CD}
|\K \pi| @> \delta^+ >> |\K \pi| \otimes |\K \pi| \\
@V \theta^{\exp} VV @VV \theta^{\exp} \otimes \theta^{\exp} V \\
|A| @> -\tdiv_{\mult} >> |A| \otimes |A| \\
@V F^{-1} VV @VV F^{-1}\otimes F^{-1} V \\
|A| @> \delta^+_{\theta_F} >> |A|\otimes |A|.
\end{CD}
\]
The upper square is commutative by Proposition \ref{prop:mult_div}, and the lower square is commutative by definition of the transfered cobracket.
For $a\in A$ we use equations \eqref{eq:mult_add}, \eqref{eq:F*tdiv}, \eqref{eq:tdej} and Proposition \ref{prop:add'_div} to compute
$$
\begin{array}{lll}
\delta^+_{\theta_F}(|a|)  & = & - (F^{-1}\otimes F^{-1})\tdiv_{\mult} (F.|a|)  \\
& =&  - F^{-1}\tdiv(F(\{ |a|,-\}_{\add}))  \\
& =& - \tdiv(\{|a|,-\}_{\add}) - \{|a|,-\}_{\add}\cdot \tJ(F^{-1})  \\
& =& \delta^{\alg}(|a|) - \{|a|,- \}_{\KKS}\cdot \tde (j(F^{-1})). \label{eq:tfalg}
\end{array}
$$
The equality $\delta^+_{\theta_F}=\delta^{\alg}$ holds true if and only if $\{|a|,- \}_{\KKS}\cdot \tde (j(F^{-1}))=0$ for every $a\in A$.
Looking at the $|A|\otimes {\bf 1}$-component of this equation, we see 
that the equation is satisfied only if $j(F^{-1}) \in |A|$ is central with respect to $\{-,-\}_{\KKS}$. In the other direction, recall that the center is a sub-coalgebra of $|A|$. Hence, if $j(F^{-1})$ is central, the expression $\{|a|,- \}_{\KKS}\cdot \tde (j(F^{-1}))$ vanishes as required.

For the second statement, if $G\in {\rm ker}(\rho)$ and $FG$ is a solution of ($\mathrm{KV}^{(0,n+1)}\mathrm{II}$), then
$$
\begin{array}{lll}
j(F^{-1}) & = & j(G) + G.j((FG)^{-1}) \\
& = &  j(G) + j((FG)^{-1}) \\
& \in & \{ |\sum_i(h(x_i) +c_i x_i) - h(-x_0)| \, c_i \in \mathbb{K}, h(z)\in \mathbb{K}[[z]]\} \\
& \subset & Z(|A|, \{-,-\}_{\rm KKS}).
\end{array}
$$
In the other direction, one uses the argument similar to the proof of Proposition \ref{thm:krvd}.
\end{proof}

\begin{rem}
One can prove Theorem \ref{thm:intro_delta_div} in Introduction by using the same argument as above. Indeed, for $F \in \taut(L)$ and $\theta_F = F^{-1} \circ \theta^{\rm exp}$ we have
$$
\delta^+_{\theta_F}(|a|) = - \tdiv(\{|a|,-\}_{\theta}) - \{|a|,-\}_{\theta}\cdot \tJ(F^{-1}).
$$
If $j(F^{-1})$ is central in the Lie algebra $(|A|, \{-,-\}_\theta)$, it is spanned by elements of the form $|x_i^k|, |(F^{-1}\log(e^{x_1}\dots e^{x_n}))^k|$ (see Remark \ref{rem:center_theta}). Since the center is a sub-coalgebra of $|A|$, the expression $\{|a|,-\}_{\theta}\cdot \tJ(F^{-1})$ vanishes, as required.
\end{rem}

\begin{rem}
 Define the set $\Theta^{(0, n+1)}$ of homomorphic expansions and the set ${\rm SolKV}^{(0, n+1)}$ of solutions of the Kashiwara-Vergne problem. There is a bijection between these two sets established by the formula $F \mapsto \theta_F = F^{-1} \circ \theta^{\rm exp}$. Furthermore, the group $\KRV_n$ acts freely and transitively on the set ${\rm SolKV}^{(0, n+1)}$, and the action is given by formula $F \mapsto F H$. The set $\Theta^{(0, n+1)}$ carries a transitive action of the group $\KRV_n^{\rm KKS}$ given by $\theta \mapsto H^{-1} \circ \theta$. The stabilizer of this action is ${\rm ker}(\rho)$ for all homomorphic expansions $\theta$. The bijection $\Theta^{(0, n+1)} \cong {\rm SolKV}^{(0, n+1)}$ is equivariant under the action of $\KRV_n$ (the action on $\Theta^{(0, n+1)}$ is induced by the inclusion $\KRV_n \subset \KRV_n^{\rm KKS}$).
\end{rem}

\begin{rem}

Let ${\rm Aut}_n^{\rm KKS} \subset \taut(L)$ be the subgroup of $\taut(L)$ preserving the Lie bialgebra structure $(|A|, \{ \cdot, \cdot\}_{\rm KKS}, \delta^{\rm alg})$. In view of Remark \ref{rem:mult_add}, it is natural to conjecture that
$$
{\rm Aut}_n^{\rm KKS} = (\KRV_n^{\rm KKS}/\{ \exp(\alpha t); \alpha \in \mathbb{K} \}) \ltimes {\rm Inn}_n,
$$
where ${\rm Inn}_n$ is the group of inner automorphisms of $A$ (which acts trivially on $|A|$) and $t: x_i \mapsto [x_i, x_0]$ generates the intersection of $\mathfrak{krv}_n^{\rm KKS}$ with inner derivations.

\end{rem}

\subsection{Topological interpretation of Duflo functions}

In this section, we apply Theorem \ref{thm:main} to establish a relation between the function $s(z)$ coming from the topologically defined double bracket $\kappa$ and the Duflo function $h(z)$ entering solutions of the Kashiwara-Vergne equations.

Recall that the function $g(z)$ in equation (\ref{eq:muF}) is the derivative of the Duflo function $h(z)$ and $h(z)$ can be chosen to lie in $z^2\K[[z]]$. Decomposing both $g$ and $h$ into a sum of even and odd parts, we obtain $g_{\rm even} = \dot{h}_{\rm odd}, g_{\rm odd} = \dot{h}_{\rm even}$.

\begin{prop}
\label{prop:Duflomu}
Let $F$ be a solution of the Kashiwara-Vergne problem of type $(0, n+1)$ for $n \geq 1$ and $h(z)$ the corresponding Duflo function.
Then, modulo the linear part,
\begin{equation}
g_{\odd}(z)=\dot{h}_{\even}(z)\equiv \frac{1}{2}\left( \frac{1}{2}+s(z) \right).
\label{eq:even2}
\end{equation}
\end{prop}

\begin{rem}
In the original Kashiwara-Vergne problem, 
the function $h(z)$ satisfies
\begin{equation}
h_{\even}(z)=-\frac{1}{2}\sum_{k\ge 1}
\frac{B_{2k}}{2k\cdot (2k)!} z^{2k}
\label{eq:even1}
\end{equation}
by \cite{AT12} Proposition 6.1. This matches the expression obtained from \eqref{eq:even2} by substituting $s(z)=z^{-1} - (1-e^{-z})^{-1}$.
\end{rem}

\begin{rem}
In fact, by inspection of the lower degree terms of the KV equations, one can also show that the linear term of $g(z)$ is $-(B_2/4)z=-(1/24)z$.
\end{rem}

\begin{proof}[Proof of Proposition \ref{prop:Duflomu}]
Recall that the double bracket $\kappa$ and operation $\mu$ satisfy the involutivity property (see \cite{KK15} Proposition 3.2.7):
$$
\{ -, -\} \circ \mu =0: \mathbb{K}\pi \to |\mathbb{K}\pi| \otimes  \mathbb{K}\pi \to \mathbb{K}\pi.
$$
Let $F$ be a solution of the Kashiwara-Vergne problem of type $(0, n+1)$. By applying expansion $\theta_F$, we obtain
\begin{equation}
\{-,-\}_{\add}\circ \mu_{\theta_F} = 0: A \to |A|\otimes A \to A.
\label{eq:invol1}
\end{equation}
Since $\{ -, 
-\}_s$ vanishes on $|A| \otimes A$, we can replace the bracket $\{-,-\}_{\add}$ by the bracket $\{-,-\}_{\rm KKS}$. 
Furthermore, the graded quotient of the equation (\ref{eq:invol1}) implies 
$\{-,-\}_{\KKS}\circ\mu^{\alg} = 0$. By using (\ref{eq:muF}), 
we obtain
$$
\{|s''|, as'\}_{\KKS} - \{|s''a|, s'\}_{\KKS} 
+ \{|g'|, ag''\}_{\KKS} - \{|g'|, g''a\}_{\KKS} = 0
$$
for any $a\in A$.
The contribution $\{|s''a|, s'\}_{\KKS}$ vanishes since $\{ |s''a|,  - \}_{\KKS}$ is a special derivation and special derivations annihilate $x_0$.
The remaining terms read
\begin{align*}
& -\dot{s}(-\ad_{x_0})(a)+\dot{s}(0)a+\dot{g}(-\ad_{x_0})(a)-\dot{g}(0)a
-(-\dot{g}(\ad_{x_0})(a)+\dot{g}(0)a) \\
=& (2\dot{g}_{\odd}(\ad_{x_0})-\dot{s}(\ad_{x_0}))(a)
+(\dot{s}(0)-2\dot{g}(0))a.
\end{align*}
Here we use that $s''\otimes s'=-1\otimes 1-s'\otimes s''$ and Lemma \ref{lem:maKKS}.
Since the last expression vanishes for all $a$,
we conclude that $2g_{\rm odd}(z)-s(z)=0$ modulo the terms of degree $\le 1$.
This proves the proposition.
\end{proof}

\appendix{}

\section{Proof of Theorem \ref{thm:center}}
\label{sec:inn}

In this appendix we deduce Theorem \ref{thm:center} from the following 
theorem. Most of this appendix will be spent for its proof. 
\begin{thm}
\label{prop:inner}
Suppose that $u\in \der(A)$ satisfies $|u(a)|=0$ for any $a\in A$.
Then $u$ is an inner derivation, i.e., there exists an element $w\in A$ such that $u(a)=[a,w]$ for any $a\in A$.
\end{thm}

Let $H$ be the linear subspace of $A$ spanned by $x_i$, $1\le i\le n$.
The linear subspace $H^{\otimes m}\subset A$ is nothing but the set of homogeneous elements of degree $m$.
In order to prove Theorem \ref{prop:inner}, 
we need Proposition \ref{prop:inner01} and Lemma \ref{lem:inner2} stated below.

\begin{prop}\label{prop:inner01} Let $x \in H\setminus\{0\}$ and
let $a \in H^{\otimes m}$ with $m\ge 1$.
If $|ax^l|=0$ for some $l\ge m-1$, then $a\in [x,A]$.
\end{prop}

To prove Proposition \ref{prop:inner01}, 
we may assume that $n\ge 2$ and $x=x_1$.\par

For a multi-index $I=(i_1,\ldots,i_m)$, where $i_k\in \{1,\ldots,n \}$,
set $p_I:=x_{i_1}\cdots x_{i_m}\in H^{\otimes m}$.
We introduce the subspaces of $H^{\otimes m}$ defined by
\begin{align*}
P&:=\mathrm{Span}_{\K} \{ p_I \mid i_1\neq 1, i_m=1 \}, \\
Q&:=\mathrm{Span}_{\K} \{ p_I \mid i_1\neq 1, i_m\neq 1 \}.
\end{align*}
Clearly, $H^{\otimes m}=xH^{\otimes (m-1)} \oplus P\oplus Q$.
Note that $P=0$ if $m=1$.

\begin{lem}
\label{lem:bx^l}
Let $m\ge 2$.
If $b\in Q$ and $|bx^l|=0$ for some $l\ge m-1$, then $b=0$.
\end{lem}
\begin{proof}
Let $\pi: H\to \K x$ be the projection defined by
$x_1\mapsto x$ and $x_i\mapsto 0$ for $i=2,\ldots,n$.
Let $\sigma:=(1,2,\ldots,l+m)\in \mathfrak{S}_{l+m}$ be the cyclic permutation.
Since $|bx^l|=0$, we have
$\sum_{i=0}^{l+m-1} \sigma^i (bx^l)=0\in H^{\otimes (l+m)}$.
Applying $1_{H}^{\otimes m} \otimes \pi^{\otimes l}$ to this equation,
we obtain $bx^l=0$, since $b\in Q$ implies
$(1_{H}^{\otimes m} \otimes \pi^{\otimes l})\sigma^i(bx^l)=0$ for $i= 1,\ldots,l+m-1$.
Hence $b=0$.
\end{proof}

\begin{proof}[Proof of Proposition \ref{prop:inner01}]
The assertion is clear if $m=1$.
In fact, $|ax^l|=0$ implies $a=0$.
\par
Let $m\ge 2$ and suppose that the assertion holds
if $\deg(a)=m-1$.
Let $a\in H^{\otimes m}$ and assume that $|ax^l|=0$ for some $l\ge m-1$.
There exist uniquely determined elements $a'\in H^{\otimes (m-1)}$, $b'\in H^{\otimes (m-1)}$ with $b'x\in P$, and $b''\in Q$ such that
$a=xa'+b'x+b''$.
We have
\[
0=|ax^l|=|(xa'+b'x+b'')x^l|=|(a'+b')x^{l+1}|+|b''x^l|.
\]
Notice that 
the first term $|(a'+b')x^{l+1}|$ is in the span of monomials $|p_I|$
containing a sequence of $l+1$ consecutive $x$'s,
while the second term $|b''x^l|$ is in the span of $|p_I|$'s
containing no such a sequence.
Therefore, $|b''x^l|=0$.
Applying Lemma \ref{lem:bx^l}, we obtain $b''=0$.
Also, we have $|(a'+b')x^{l+1}|=0$.
\par
Now we have $l+1\ge (m-1)-1$ since $l\ge m-1$.
By the inductive assumption,
$a'+b'=[x,a_1]$ for some $a_1\in H^{\otimes (m-2)}$.
Then $a=xa'+b'x=[x,a']+(a'+b')x=[x,a']+[x,a_1]x=[x,a'+a_1x]\in [x,A]$, as required.
\end{proof}

\begin{cor}\label{prop:inner1} Let $x \in H\setminus\{0\}$ and $a \in A$. 
Then $a \in [x, A]$ if and only if $| ax^l| = 0$ for any $l\ge 1$.
\end{cor}

\begin{proof}
The `only if'-part is trivial. In fact, it is clear that 
$| [x, b]x^l| = -|b [x, x^{l}]| = 0$. 
Here we remark that it suffices to show the `if'-part
for any homogeneous $a \in A$,
since $[A, A] \subset A$ is homogeneous. 
Hence Proposition \ref{prop:inner01} implies the `if'-part.
\end{proof}

\begin{lem}\label{lem:inner2} Let $x$ and $y \in H$ be linearly independent.
Suppose that an element $a \in A$ satisfies $[x^l, a] \in [y, A]$ and 
$[y^l, a] \in [x, A]$ for any $l \ge 0$. Then we have 
$a \in \K[[x]] + \K[[y]]$. 
\end{lem}

\begin{proof}
We may assume $a \in H^{\otimes d}$ for some $d \ge 1$.
Since $\K x^d \cap yH^{\otimes (d-1)} = 0$, we have a decomposition 
$$
H^{\otimes d} = \K x^d \oplus y H^{\otimes (d-1)} \oplus V
$$
for some linear subspace $V \subset H^{\otimes d}$. 
In particular, we have $a = \lambda x^d + ya' + a''$ for some $\lambda \in \K$, 
$a' \in H^{\otimes(d-1)}$ and $a'' \in V$. 
From the assumption, there exists some $b \in H^{\otimes(2d-1)}$ such that 
$[x^d, a] = [y,b]$. We have $b = -x^db_0 + yb' + b''$ 
for some $b_0 \in H^{\otimes(d-1)}$, $b' \in H^{\otimes(2d-2)}$ and 
$b'' \in V\otimes H^{\otimes(d-1)}$. Then the equation $[x^d, a] = [y,b]$ implies
$$
x^d(ya' + a'') - (ya' + a'')x^d = yb + x^db_0y - (yb' + b'')y
$$
Looking at its $x^dH^{\otimes d}$-component in 
the decomposition
$H^{\otimes 2d} = x^dH^{\otimes d} \oplus yH^{\otimes(2d-1)} 
\oplus V\otimes H^{\otimes d}$,
we obtain $x^d(ya'+ a'') = x^db_0y$, so that $ya' + a'' = b_0y$. 
So we conclude $a = \lambda x^d + b_0y$. 
Similarly the assumption $[y^d, a] \in [x, A]$ 
implies $a = \mu y^d + c_0x$ for some $\mu \in \K$ and 
$c_0 \in H^{\otimes (d-1)}$. Hence we obtain
$a = \lambda x^d + \mu y^d$. This proves Lemma \ref{lem:inner2}.
\end{proof}

\begin{rem}
Note that if $a \in \K[[x]] + \K[[y]]$ then $[x^l, a] \in [y, A]$ and 
$[y^l, a] \in [x, A]$ for any $l \ge 0$. For example,
$$
[x^l, y^m] = \left[x, \sum_{i=0}^{l-1} x^i y^m x^{l-1-i}\right].
$$
\end{rem}

\begin{proof}[Proof of Proposition \ref{prop:inner}]
If $n = 1$, then $| A| = A$, so that the theorem is trivial.
Hence we may assume $n \geq 2$. Since
$
0 = | u({x_i}^{l+1})| = \sum^l_{k=0}| {x_i}^ku(x_i){x_i}^{l-k}| 
= (l+1)| u(x_i){x_i}^l|
$
for any $l \geq 0$, we have some $a_i \in A_{\ge1}$ such that 
$
u(x_i) = [x_i, a_i]
$
by Corollary \ref{prop:inner1}. 
Let $i$ and $j \in \{1,2,\dots, n\}$ be distinct. Then, for any $p, q \ge 1$, 
we have 
$$
0 = | u({x_i}^p{x_j}^q)| = 
| [{x_i}^p, a_i]{x_j}^q| + | {x_i}^p[{x_j}^q, a_j]|
= |[{x_i}^p, a_i-a_j]{x_j}^q| = |{x_i}^p [a_i-a_j, {x_j}^q]|.
$$
This implies $[{x_i}^p, a_i-a_j] \in [x_j, A]$ and $[{x_j}^q, a_i-a_j] \in [x_i, A]$ 
by Corollary \ref{prop:inner1}. 
Hence, by Lemma \ref{lem:inner2}, we have $a_i - a_j \in \K[[x_i]]+\K[[x_j]]$, 
namely, we have unique formal series $f_{ij}(x_j) \in \K[[x_j]]$ and 
$f_{ji}(x_i) \in \K[[x_i]]$ such that 
$a_i - a_j = f_{ij}(x_j) - f_{ji}(x_i)$ and $f_{ij}(0) = f_{ji}(0) = 0$. 
If $n = 2$, then $w := a_1 + f_{21}(x_1) = a_2 + f_{12}(x_2)$ satisfies 
$u(a) = [a, w]$ for any $a \in A$. \par
For the rest of this proof, we assume $n \geq 3$. 
Let $i, j$ and $k \in \{1,2,\dots, n\}$ be distinct. Then 
$$
\aligned
0 &= (a_i-a_j) + (a_j-a_k) + (a_k-a_i)\\
&=  f_{ij}(x_j) - f_{ji}(x_i) +  f_{jk}(x_k) - f_{kj}(x_j) +  f_{ki}(x_i) - f_{ik}(x_k)
\endaligned
$$
Since $f_{ji}(0) = f_{ki}(0) = 0$, we have $f_{ji}(x_i) = f_{ki}(x_i)$.
Hence $w := a_i+f_{ji}(x_i) \in A$ is independent of the choice of $i$ and $j$, 
so that $u(a) = [a, w]$ for any $a \in A$.
\end{proof}

\begin{proof}[Proof of Theorem \ref{thm:center}]
From Lemma \ref{prop:KKS}, the center includes $|\K[[x_i]]|$ for any $i=0,1,\ldots,n$.
\par
Let $a \in A$ be a homogeneous element of degree $m\ge 1$
satisfying $\{|a|, |b|\}_{\KKS} = 0$ for any $|b| \in |A|$.
This means $|\{|a|, b\}_{\KKS}| = 0$ for any $b \in A$.
From Theorem \ref{prop:inner}, there exists an element $w \in A$ 
such that $\{|a|, b\}_{\KKS} = [b, w]$ for any $b \in A$.
By Lemma \ref{prop:KKS}, we have 
$[x_0, w] = \{|a|, x_0\}_{\KKS} = 0$. 
It is easy to check that for any non-zero element $r\in A$ of degree $1$, any element of $A$ commuting with $r$ is a formal power series in $r$. 
Hence we have $w \in \K[[x_0]]$, and so 
$\{|a|, -\}_{\KKS} = (c_1{x_1}^{m-1} + w, \dots, c_n{x_n}^{m-1} + w)\in \sder(A)$
for some $c_i \in \K$. By Lemma \ref{lem:m-times}, we obtain 
$m|a| = |N(|a|)|=- \sum_i|c_i{x_i}^m + x_iw| = |x_0w| - \sum_i|c_i{x_i}^m|
\in \sum^n_{i=0}|\K[[x_i]]|$. This proves the theorem. 
\end{proof}

\begin{rem}
The same method allows to compute the center of the \emph{completed} Goldman Lie algebra for any compact connected oriented surface with non-empty boundary.
The key point is to use expansions to transfer the Goldman bracket to a Lie algebra structure on the space of cyclic words, as we saw in the case of genus zero (see diagram \eqref{eq:GoKKS}).
For surfaces of genus $g \geq 1$, one can also use the technique of \cite{CBEG07}, see Corollary 8.6.2.
\end{rem}

\end{document}